\documentclass[a4paper] {article}
[12pt]

\makeatletter
\renewcommand*\l@section{\@dottedtocline{1}{1.5em}{2.3em}}
\makeatother

\usepackage{amsfonts}
\usepackage{amssymb}
\usepackage[T1]{fontenc}

\usepackage{tikz}
\usetikzlibrary{calc}

\usepackage{CJK}
\usepackage{amsmath}
 
\usepackage{amsfonts}
\usepackage{amssymb}
\usepackage{amsthm}
\usepackage{amssymb}
\usepackage{enumerate}
\usepackage[calc]{picture}
\usepackage[all,cmtip]{xy}

\usepackage[mathscr]{eucal}
\usepackage{eqlist}

\usepackage{color}
\usepackage{abstract} 
\usepackage[T1]{fontenc}
 
\usepackage[top=3.8cm, bottom=3.8cm, left=3.5 cm, right= 3.5cm]{geometry}

\theoremstyle{plain}
\newtheorem{theorem}{Theorem}
\newtheorem{proposition}[theorem]{Proposition}
\newtheorem{lemma}[theorem]{Lemma}

\newtheorem{example}[theorem]{Example}
\newtheorem{corollary}[theorem]{Corollary}

\theoremstyle{definition}
\newtheorem{definition}{Definition}

\usepackage{etoolbox}
\newtheoremstyle{myrem}
 {3pt}
 {3pt}
 {\normalsize}
 { }
 {\itshape}
 {:}
 { }
 {}

 \theoremstyle{myrem}
 \newtheorem{remark}{Remark}
 \appto\remark{\leftskip\parindent}
 \appto\remark{\rightskip\parindent}

\numberwithin{equation}{section}
\numberwithin{theorem}{section}

\begin{document}

~~~
 
 \vspace{1.8cm}

\begin{center}
{\Large {\textbf { The Embedded Homology of Hypergraph Pairs 
}}}
 \vspace{1.8cm}\\

Shiquan Ren*,  Jie  Wu*,   Mengmeng  Zhang*
 \footnotetext{* first authors.}

\begin{abstract}
 In  this paper,  we generalize the embedded homology  groups of  hypergraphs  initially given  in \cite{hg1}   and study the   relative embedded homology groups of  hypergraph pairs.   We  prove some long exact sequences   as  well as  a Mayer-Vietoris sequence  for  the  relative  embedded homology groups  of  hypergraph pairs.   Moreover,  we  briefly  discuss the two-dimensional persistence  for  the relative embedded homology groups  of  hypergraph pairs.
\end{abstract}

\end{center}

\begin{quote}


\end{quote}

\section{Introduction}

The  relative homology groups  of   chain  complex pairs   is a generalization of the homology groups   of  a  single  chain complex.   In classical homology theory,  relative simplicial homology groups are used to study   simplicial complex  pairs  (cf. \cite[p. 115]{hatcher} and \cite[Chapter~3]{eat}).    In \cite{digraph1},  A. Grigor'yan, R. Jimenez, Y. Muranov and S.-T. Yau defined the relative path homology  groups  for digraphs and studied the Eilenberg-Steenrod axioms for the path homology.

In various disciplines of sciences and technologies, hypergraphs are   models for complex networks.   A  hypergraph can be interpreted as a  "virtual"  simplicial complex with some  non-maximal   faces missing.   The    simplicial homology  methods  can be applied to investigate   hypergraphs  (cf.  \cite{hg1,parks,pers}).

Let $V$  be a finite set with a total order $\prec$.
The elements of $V$ are called {\it vertices}.
Let $\Delta[V]$ be the standard (abstract) simplicial complex  with  vertices from $V$  (cf.  \cite[pp. 103-105]{hatcher} and \cite[Example~1.1]{hg1}).
The set of  the simplices  in $\Delta[V]$
is given by  all the non-empty subsets of $V$.
For convenience,  we   just   use   the (abstract) simplicial complex  $\Delta[V]$     to denote   the collection of all the non-empty subsets of $V$.

A {\it hypergraph} $\mathcal{H}$ with vertices from $V$
 is a subset of $\Delta[V]$  (cf.  \cite[p. 480]{hg1}) 
 An element of $\mathcal{H}$ is called a {\it hyperedge}  of  $\mathcal{H}$.  For   any hyperedge  $\sigma\in\mathcal{H}$,  we can write  $\sigma$  in the form
 \begin{eqnarray*}
 \sigma=\{v_0,v_1,\cdots,v_n\}
 \end{eqnarray*}
 uniquely
   where $n\geq 0$  and $v_0\prec v_1\prec\cdots \prec v_n$  are distinct vertices in $V$.    We define   $n$  to be the dimension of $\sigma$.    We  say that $\sigma$ is of {\it dimension} $n$  or  $\sigma$ is an {\it $n$-hyperedge}.  We  write $\dim \sigma=n$.

Let $\mathcal{H}$  be a hypergraph with vertices from $V$.   If for any  hyperedge $\sigma\in\mathcal{H}$ and any non-empty subset $\tau\subseteq\sigma$  we  always have $\tau\in\mathcal{H}$,  then we call $\mathcal{H}$    an {\it (abstract)  simplicial complex}.       A  hyperedge  of a  simplicial complex is called a {\it simplex}.

   In 1991, A.D. Parks and S.L. Lipscomb \cite{parks}  used  simplicial complexes  and  simplicial homology  to study  hypergraphs.  Let  $\mathcal{H}$  be  a  hypergraph.  A.D. Parks and S.L. Lipscomb \cite{parks}   defined the {\it associated simplicial complex}  $\Delta\mathcal{H}$  of  $\mathcal{H}$  to  be the simplicial complex obtained from $\mathcal{H}$  by adding all the missing faces.
   Specifically, by  \cite{parks},  the associated simplicial complex    $\Delta\mathcal{H}$   of $\mathcal{H}$  is given by
\begin{eqnarray}\label{eq-ass}
\Delta\mathcal{H}=\{\sigma\in \Delta[V]\mid \sigma\subseteq \tau{\rm ~for~some~}\tau\in\mathcal{H}\}.
\end{eqnarray}
It can be  seen that $\Delta\mathcal{H}$ is   the smallest simplicial complex whose the set of the simplices contains all the hyperedges of $\mathcal{H}$.

In 2020,  supplementary to the associated simplicial complex $\Delta\mathcal{H}$  of $\mathcal{H}$     defined in   \cite{parks},         the  {\it lower-associated simplicial complex}  $\delta\mathcal{H}$  of $\mathcal{H}$  is defined by  S.  Ren   \cite[Subsection~2.1]{pers}  as
\begin{eqnarray}\label{eq-lca}
\delta\mathcal{H}=\{\sigma\in\mathcal{H}\mid {\rm ~for~any~}\tau\subseteq\sigma {\rm ~we~have~} \tau\in\mathcal{H}\}.
\end{eqnarray}
It can be seen that $\delta\mathcal{H}$ is the largest simplicial complex  whose the set  of
the  simplices is   contained  in the set  of  the hyperedges of  $\mathcal{H}$.   In  other words,   $\delta\mathcal{H}$  is  the  simplicial complex  obtained from $\mathcal{H}$ by removing all the hyperedges in $\mathcal{H}$ with missing faces.

Let $G$  be an abelian group.  We have a chain complex $C_*(\Delta\mathcal{H})$ with its coefficients in  $G$.
We let $\partial_*$ be the boundary map of $C_*(\Delta\mathcal{H})$.   For each $n\geq 0$,  we write $G(\mathcal{H})_n$  for  the  abelian group consisting  of all the formal linear combinations of the $n$-hyperedges in $\mathcal{H}$  with coefficients in $G$.   We have  a  graded abelian  sub-group
   \begin{eqnarray*}
   G(\mathcal{H})_*=\{G(\mathcal{H})_n\}_{n\geq 0}
   \end{eqnarray*}
   of $C_*(\Delta\mathcal{H})$.

   In 2019,  S. Bressan, J. Li, S. Ren and J. Wu \cite{hg1} applied the path homology methods by  A. Grigor'yan, Y. Lin, Y. Muranov and S.-T. Yau \cite{lin1,lin2} to hypergraphs.
They defined in \cite{hg1}  the {\it infimum chain complex} ${\rm Inf}(\mathcal{H})$  and the  {\it supremum chain complex}  ${\rm Sup}(\mathcal{H})$  respectively as the largest sub-chain complex of $C_*(\Delta\mathcal{H})$ contained in $G(\mathcal{H})_*$ and the smallest sub-chain complex of  $C_*(\Delta\mathcal{H})$ containing  $G(\mathcal{H})_*$.    For each $n\geq 0$,   it  is given  in \cite{hg1}  that
\begin{eqnarray*}
{\rm Inf}_n(\mathcal{H})=G(\mathcal{H})_n\cap \partial_n^{-1}G(\mathcal{H})_{n-1},  ~~~
{\rm Sup}_n(\mathcal{H})=G(\mathcal{H})_n+  \partial_{n+1}G(\mathcal{H})_{n+1}.
\end{eqnarray*}
It  is  deduced from  \cite[Proposition~2.4]{hg1}  that the   canonical  inclusion $\iota$ of ${\rm Inf}_*(\mathcal{H})$  into ${\rm Sup}_*(\mathcal{H})$  induces an isomorphism of homology groups.  The   embedded homology of $\mathcal{H}$ is defined in \cite{hg1}  as $H_*({\rm Inf}(\mathcal{H}))\cong H_*({\rm Sup}(\mathcal{H}))$,  which is denoted as $H_*(\mathcal{H})$ for short.

In this paper,   we generalize the embedded homology  groups  of  hypergraphs  initially given   in \cite{hg1}     and  define  the  relative  embedded homology  groups of hypergraph pairs.  Given a hypergraph pair $(\mathcal{H},\mathcal{A})$  where $\mathcal{A}$  is a sub-hypergraph of $\mathcal{H}$ (cf.  Definition~\ref{def-mfweq}  and   Definition~\ref{def-cmahg}),  we  prove in   Section~\ref{s3x} that  the  homology of the quotient infimum chain complex ${\rm Inf}(\mathcal{H})/{\rm Inf}(A)$   is canonically   isomorphic  to the homology of the quotient supremum chain complex ${\rm Sup}(\mathcal{H})/{\rm Sup}(\mathcal{A})$.   We  define this homology as the      relative embedded homology groups     of    $(\mathcal{H},\mathcal{A})$.
  We  prove some long exact sequences  for   the embedded  homology  of   hypergraph pairs $(\mathcal{H},\mathcal{A})$    in   Theorem~\ref{pr-xqy}  (main result I).  Then we derive Corollary~\ref{th1}   from  Theorem~\ref{pr-xqy}.   We  also  prove a Mayer-Vietoris sequence for the  relative embedded homology  groups  of  the hypergraph pair  $(\mathcal{H},\mathcal{A})$    in Theorem~\ref{th-3.888} (main result II).   Then we  derive Corollary~\ref{co-qmboz}  from Theorem~\ref{th-3.888}.   Supplementary  to  Theorem~\ref{pr-xqy}   and  Theorem~\ref{th-3.888},    we briefly discuss the two-dimensional  persistence of the  relative embedded homology  groups  of  hypergraph pairs for the purpose of potential applications,  by the end of this paper.

We  give some preliminaries in Section~\ref{s2x}.  We  define the  relative embedded homology  groups of  hypergraph pairs in Section~\ref{s3x}.  We  prove some auxiliary results in Section~\ref{s4x}.  As a by-product,  we give a cell structure for hypergraphs  by using the   relative  embedded homology groups,    in Section~\ref{s5x}.   We prove Main Result I in  Theorem~\ref{pr-xqy},  Section~\ref{s6x} and prove Main Result II  in  Theorem~\ref{th-3.888},  Section~\ref{s7x}.   Finally,   we give  a brief discussion on  the two-dimensional persistence of the relative embedded homology groups,  in Section~\ref{s8x}.

\section{Preliminaries}\label{s2x}

In  Subsection~\ref{ss2.1},  we  give some algebraic preliminaries on graded abelian subgroups of chain complexes and their embedded homology.  In  Subsection~\ref{ss2.2},  we give some topological preliminaries on   hypergraphs  and  their (lower-)associated simplicial complexes.

\subsection{Graded abelian subgroups of chain complexes  and homology}\label{ss2.1}

Let $C=\{C_n,\partial_n\}_{n\geq 0}$ be a chain complex.   Let $D=\{D_n\}_{n\geq 0}$ and $D'=\{D'_n\}_{n\geq 0}$ be  graded abelian subgroups of $C$  with $D'_n\subseteq D_n$ for each $n\geq 0$.   Let
\begin{eqnarray*}
&{\rm  Inf}_n(D,C)=D_n\cap \partial_n^{-1}D_{n-1}, ~~~
{\rm  Inf}_n(D',C) = D'_n\cap \partial_n^{-1}D'_{n-1}, \\
&{\rm  Sup}_n(D,C) = D_n+ \partial_{n+1} D_{n+1}, ~~~
{\rm  Sup}_n(D',C) = D'_n+ \partial_{n+1} D'_{n+1}   
\end{eqnarray*}
for each $n\geq  0$.   The {\it infimum  chain complex}  of $D$  in $C$ and the  {\it infimum  chain complex} of $D'$  in $C$ are 
\begin{eqnarray*}
{\rm  Inf}(D,C)&=&\{{\rm  Inf}_n(D,C), \partial_n\mid_{{\rm  Inf}_n(D,C)}\}_{n\geq 0}, \\
{\rm  Inf}(D',C)&=&\{{\rm  Inf}_n(D',C), \partial_n\mid_{{\rm  Inf}_n(D',C)}\}_{n\geq 0}.  
\end{eqnarray*}
The {\it  supremum  chain complex}  of $D$ in $C$ and     the  {\it supremum  chain complex} of $D'$  in $C$ are
\begin{eqnarray*}
{\rm  Sup}(D,C)&=&\{{\rm  Sup}_n(D,C), \partial_n\mid_{{\rm  Sup}_n(D,C)}\}_{n\geq 0}, \\
{\rm  Sup}(D',C)&=&\{{\rm  Sup}_n(D',C), \partial_n\mid_{{\rm  Sup}_n(D',C)}\}_{n\geq 0}. 
\end{eqnarray*}
 The  next lemma follows from   \cite[Proposition~2.1 and  Proposition~2.2]{hg1}.  
\begin{lemma}\label{le-zopqs|}
  For each $n\geq 0$,    the boundary maps $\partial_n\mid_{{\rm  Inf}_n(D,C)}$, $\partial_n\mid_{{\rm  Inf}_n(D',C)}$, $\partial_n\mid_{{\rm  Sup}_n(D,C)}$  and  $\partial_n\mid_{{\rm  Sup}_n(D',C)}$  are well-defined.  Thus ${\rm  Inf}(D,C)$, ${\rm  Inf}(D',C)$,  ${\rm  Sup}(D,C)$   and ${\rm  Sup}(D',C)$  are  all  sub-chain complexes  of  $C_*(\Delta\mathcal{H})$.  
 \end{lemma}
 The  next lemma follows from  \cite[Proposition~2.4]{hg1}.  
 \begin{lemma} 
 The canonical inclusion of     ${\rm  Inf}(D,C)$  (resp.   ${\rm  Inf}(D',C)$)  into    ${\rm  Sup}(D,C)$  (resp. ${\rm  Sup}(D',C)$)  induces an isomorphism between the homology groups $H_*({\rm  Inf}(D,C))$  (resp. $H_*({\rm  Inf}(D',C))$)  and  $H_*({\rm  Sup}(D,C))$  (resp. $H_*({\rm  Sup}(D',C))$).  
 \end{lemma}

Since $D'_n\subseteq  D_n$  for  each  $n\geq  0$,  it  holds  that  
$\text{Inf}(D',C)$ is a sub-chain complex of $\text{Inf}(D,C)$  and $\text{Sup}(D',C)$ is a sub-chain complex of $\text{Sup}(D,C)$.   The following  diagram commutes 
{\footnotesize \begin{eqnarray}
\xymatrix{
0\ar[d] && 0\ar[d] &\\
{\rm Inf}(D',C)\ar[r]\ar[d] & D'\ar[d] \ar[r] &{\rm Sup}(D',C)\ar[d] \ar[r] &C\\
{\rm Inf}(D,C)\ar[r]\ar[d] & D \ar[r] &{\rm Sup}(D,C) \ar[r]\ar[d] &C\\
{\rm Inf}(D,C)/{\rm Inf}(D',C)\ar[d] \ar[rr]^{\iota} && {\rm Sup}(D,C)/{\rm Sup}(D',C)  \ar[d] &\\
0 && 0.  &
}
\label{eq-lavs}
\end{eqnarray}}In the   diagram  (\ref{eq-lavs}),  both of the followings are satisfied:
\begin{enumerate}[(i).]
\item
 all  the  horizontal  maps in the second  row  and   all  the  horizontal  maps  in  the third  row are inclusions;   
 \item
 all  the  vertical maps in  the first column as  well  as  all  the  vertical maps in  the third column give  a  short exact sequence  of  chain complexes.  
  \end{enumerate}
  The  next lemma shows that the chain map $\iota$  in the diagram  (\ref{eq-lavs})  induces an isomorphism of homology groups. 
\begin{lemma}\label{le-1.1}
The chain map 
\begin{eqnarray*}
\iota:     {\rm Inf}(D,C)/{\rm Inf}(D',C)\longrightarrow {\rm Sup}(D,C)/{\rm Sup}(D',C)
\end{eqnarray*}
 induces an isomorphism of homology
 \begin{eqnarray}\label{eq-qozj}
\iota_*:     H_*({\rm Inf}(D,C)/{\rm Inf}(D',C))\overset{\cong}{\longrightarrow}H_*( {\rm Sup}(D,C)/{\rm Sup}(D',C)). 
\end{eqnarray}
\end{lemma}
\begin{proof}  
Let $n\geq 0$.  
By  Lemma~\ref{le-zopqs|}  or  \cite[Proposition~2.4]{hg1},  we  have   isomorphisms
\begin{eqnarray*}
 & H_n({\rm Inf}(D',C)) \overset{\cong}{\longrightarrow}   H_n({\rm Sup}(D',C)), \\
&  H_n({\rm Inf}(D,C)) \overset{\cong}{\longrightarrow}   H_n({\rm Sup}(D,C)),\\
 &  H_{n-1}({\rm Inf}(D',C)) \overset{\cong}{\longrightarrow}   H_{n-1}({\rm Sup}(D',C)), \\
 & H_{n-1}({\rm Inf}(D,C)) \overset{\cong}{\longrightarrow}   H_{n-1}({\rm Sup}(D,C)).  
\end{eqnarray*}
  Applying the homology functor    respectively   to the first column and  to  the third column  in the  diagram  (\ref{eq-lavs}),  we have two long exact sequences of homology groups which are given in the two rows of  the following   diagram
{\footnotesize \begin{eqnarray}
\xymatrix{
\cdots \ar[r] & H_n({\rm Inf}(D',C))\ar[r]\ar[d]^{\cong} &H_n({\rm Inf}(D,C)) \ar[r]\ar[d]^{\cong}  &&\\
\cdots \ar[r] & H_n({\rm Sup}(D',C))\ar[r] &H_n({\rm Sup}(D,C)) \ar[r] &&\\
 \ar[r]& H_n({\rm Inf}(D,C)/{\rm Inf}(D',C))\ar[d]^{}\ar[r] &H_{n-1}({\rm Inf}(D',C))\ar[r]\ar[d]^{\cong}&  \\
\ar[r]& H_n({\rm Sup}(D,C)/{\rm Sup}(D',C))\ar[r] &H_{n-1}({\rm Sup}(D',C))\ar[r]  &
\\
  &H_{n-1}({\rm Inf}(D,C)) \ar[r]\ar[d]^{\cong }&\cdots \\
  &H_{n-1}({\rm Sup}(D,C))\ar[r]&\cdots  }
\label{eq-pqmf}
\end{eqnarray}}The  diagram (\ref{eq-pqmf}) commutes.  Moreover,  the first, the second,  the fourth, and the  fifth vertical maps in the    diagram  (\ref{eq-pqmf})   are isomorphisms.   Applying the Five Lemma  to   the   diagram  (\ref{eq-pqmf}), we obtain that (\ref{eq-qozj})  is an isomorphism.  
\end{proof}

\subsection{Hypergraphs and simplicial complexes}\label{ss2.2}

Let $V$  be a finite set with a total order $\prec$.   
\begin{definition}(cf.  \cite[Chapter~17]{berge}  and \cite{hg1,parks,pers})\label{def-111a}
A  {\it hypergraph}   with vertices from $V$  is a collection of non-empty subsets of $V$.  An element in a hypergraph   is called a {\it hyperedge},  which is just  a non-empty subset of $V$. 
\end{definition}
 Let $\mathcal{H}$ be a hypergraph with  vertices from $V$.  
 \begin{definition} \label{def-mfweq}
  A {\it sub-hypergraph}  $\mathcal{A}$  of $\mathcal{H}$,  denoted as $\mathcal{A}\subseteq\mathcal{H}$,   is a hypergraph  with vertices from $V$ such that  for any $\sigma\in\mathcal{A}$,  we always have $\sigma\in\mathcal{H}$. 
\end{definition}
\begin{definition}(cf.  \cite[p. 107 and  p. 177]{hatcher}  and  \cite[Chapter~1, Section~2]{eat})\label{def-111b}
  An   {\it (abstract) simplicial complex}   with vertices from $V$  is a hypergraph $\mathcal{K}$ with vertices from $V$   such that for any $\sigma\in\mathcal{K}$  and any non-empty subset $\tau\subseteq \sigma$ we always have $\tau\in\mathcal{K}$.    
\label{def-pqmpxd111}
\end{definition}
\begin{definition}(cf.  \cite{hg1, parks} and  \cite[Subsection~2.1]{pers})
\label{def-zoqq}
\begin{enumerate}[(i).]
\item
 The {\it associated simplicial complex} $\Delta\mathcal{H}$   is  the  simplicial complex  (\ref{eq-ass})  whose the set of the simplices consists of all the non-empty subsets of   each   hyperedge  in $\mathcal{H}$; 
 \item
 The {\it lower-associated simplicial complex} $\delta\mathcal{H}$  is the simplicial complex (\ref{eq-lca})  whose the set of the simplices consists of all the hyperedges $\sigma$ in $\mathcal{H}$ such that all the non-empty subsets of $\sigma$ are still hyperedges in $\mathcal{H}$.  
 \end{enumerate}
\end{definition}
Let $V$  and $V'$ be  finite sets  with   total orders $\prec$.   Let $\mathcal{H}$  be a hypergraph  with vertices from  $V$ and let $\mathcal{H}'$ be a hypergraph with vertices from $V'$.  
\begin{lemma}\label{le-int-union}
(i). 
$\delta(\mathcal{H}\cap\mathcal{H}')=(\delta\mathcal{H})\cap(\delta\mathcal{H}')$; 
(ii).  
$\delta(\mathcal{H}\cup\mathcal{H}')\supseteq (\delta\mathcal{H})\cup(\delta\mathcal{H}')$; 
(iii).  
$\Delta(\mathcal{H}\cap\mathcal{H}')\subseteq (\Delta\mathcal{H})\cap(\Delta\mathcal{H}')$; 
(iv).  
$\Delta(\mathcal{H}\cup\mathcal{H}')= (\Delta\mathcal{H})\cup(\Delta\mathcal{H}')$.  
\end{lemma}
\begin{proof}
The proof is a direct verification  by (\ref{eq-ass}),  (\ref{eq-lca})    and Definition~\ref{def-pqmpxd111}. 
\end{proof}
\begin{definition}(cf.  \cite[Subsection~3.1]{hg1}  and  \cite{parks,pers})
\label{def-zoqqq} 
A {\it morphism}  $\varphi: \mathcal{H}\longrightarrow \mathcal{H}'$  of hypergraphs from $\mathcal{H}$ to $\mathcal{H}'$  is a map $\varphi: V\longrightarrow V'$  such that for any hyperedge $\sigma=\{v_0,v_1,\ldots,v_n\}$ of $\mathcal{H}$,  its image
$\varphi(\sigma)= \{\varphi(v_0),\varphi(v_1),\ldots, \varphi(v_n)\}$
  gives a hyperedge  of $\mathcal{H}'$.  
\end{definition}
\begin{remark}\label{remark-qmfj}
In Definition~\ref{def-zoqqq},  $v_0$, $v_1$, $\ldots$, $v_n$ are distinct vertices in $V$ with $v_0\prec v_1\prec\ldots \prec v_n$ while the vertices  $\varphi(v_0)$, $\varphi(v_1)$, $\ldots$, $\varphi(v_n)$  in $V'$ may not be  distinct and there may exist $0\leq i<j\leq n$  such that $\varphi(v_j)\prec \varphi(v_i)$.   Nevertheless,  we  can delete the repeated vertices in $\varphi(v_0)$, $\varphi(v_1)$, $\ldots$, $\varphi(v_n)$ and then rearrange the order.  
\end{remark}
\begin{definition}(cf.  \cite[p. 107 and  p. 177]{hatcher}  and  \cite[Chapter~1, Section~2]{eat}) 
Let $\mathcal{K}$  be a simplicial complex  with vertices from  $V$ and let $\mathcal{K}'$ be a simplicial complex with vertices from $V'$.   
 A  {\it simplicial map}  from $\mathcal{K}$  to $\mathcal{K}'$   is just a morphism $\varphi:  \mathcal{K}\longrightarrow \mathcal{K}'$  of hypergraphs. 
\label{def-pqmpxd}
\end{definition}
\begin{definition}(cf.  \cite{hg1,parks} and  \cite[Subsection~2.3]{pers})\label{def-zoq}
   Let $\varphi: \mathcal{H}\longrightarrow \mathcal{H}'$  be  a  morphism  of  hypergraphs.  
\begin{enumerate}[(i).]
\item
The  {\it associated  simplicial map}  $\Delta\varphi$  of $\varphi$ is a simplicial map 
$\Delta\varphi:  \Delta\mathcal{H}\longrightarrow \Delta\mathcal{H}'$  given by $\varphi(\sigma)= \{\varphi(v_0),\varphi(v_1),\ldots, \varphi(v_n)\}
$   for any simplex $\sigma\in\Delta\mathcal{H}$  with  $\sigma=\{v_0,v_1\ldots,v_n\}$;   
\item
The  {\it  lower-associated  simplicial map}  $\delta\varphi$  of $\varphi$ is a simplicial map 
$\delta\varphi:  \delta\mathcal{H}\longrightarrow \delta\mathcal{H}'$  given by  $\varphi(\sigma)= \{\varphi(v_0),\varphi(v_1),\ldots, \varphi(v_n)\}
$   for any simplex $\sigma\in\delta\mathcal{H}$  with  $\sigma=\{v_0,v_1\ldots,v_n\}$.  
\end{enumerate}
\end{definition}
\begin{remark}
Remark~\ref{remark-qmfj}  is applicable to   Definition~\ref{def-zoq}~(i) and (ii).  
\end{remark}
Let $\mathfrak{H}$  be the category of hypergraphs whose objects are hypergraphs and whose morphisms are morphisms of hypergraphs. 
Let $\mathfrak{K}$  be the category of simplicial complexes  whose objects are simplicial complexes and whose morphisms are simplicial maps.

\begin{lemma}\label{le-2.apqir}
We  have a (covariant)  functor $\Delta: \mathfrak{H}\longrightarrow \mathfrak{K}$  and a   (covariant)  functor    $\delta: \mathfrak{H}\longrightarrow \mathfrak{K}$.  
\end{lemma}

\begin{proof}
With the helps  of  Definition~\ref{def-111a},  Definition~\ref{def-111b},    Definition~\ref{def-zoqq}, Definition~\ref{def-zoqqq},  Definition~\ref{def-pqmpxd}  and  Definition~\ref{def-zoq},   it  can be verified that  
\begin{eqnarray*}
\Delta({\rm id}_\mathcal{H})={\rm  id} _{\Delta \mathcal{H}}
\end{eqnarray*}
for any  hypergraph $\mathcal{H}\in{\rm Obj}(\mathfrak{H})$  and  
\begin{eqnarray*}
\Delta(\psi\circ \varphi)= \Delta\psi\circ \Delta\varphi
\end{eqnarray*}
for any morphisms $\varphi: \mathcal{H}\longrightarrow \mathcal{H}'$  and $\psi: \mathcal{H}'\longrightarrow \mathcal{H}''$  of  hypergraphs.  Thus $\Delta: \mathfrak{H}\longrightarrow \mathfrak{K}$ is a (covariant) functor.  Similarly,  we can prove that $\delta: \mathfrak{H}\longrightarrow \mathfrak{K}$ is a (covariant) functor.  
\end{proof}

\section{The  relative  embedded  homology  groups  for  hypergraph  pairs  and  examples}\label{s3x}

In Subsection~\ref{ss3.1},  we  define the relative embedded homology groups for hypergraph pairs.  In Subsection~\ref{ss3.2},  we give some examples.

\subsection{The relative embedded homology groups for hypergraph pairs}\label{ss3.1}

\begin{definition}\label{def-cmahg}
Let $\mathcal{H}$ be a hypergraph.   Let  $\mathcal{A}$ be  a sub-hypergraph of $\mathcal{H}$.  We  call the pair $(\mathcal{H},\mathcal{A})$  a {\it hypergraph pair}.  In addition,  if  both $\mathcal{H}$  and $\mathcal{A}$  are   simplicial  complexes,  then we  call the hypergraph pair  $(\mathcal{H},\mathcal{A})$    a  {\it  simplicial complex pair}.  
\end{definition}

 Firstly,  let $(\mathcal{H},\mathcal{A})$  be  a   hypergraph pair.  
The canonical  inclusion $i: \mathcal{A}\longrightarrow\mathcal{H}$ of hypergraphs  is  a  morphism  of  hypergraphs.   By  Lemma~\ref{le-2.apqir},   the morphism  $i$  induces an inclusion $\delta i:  \delta\mathcal{A}\longrightarrow \delta\mathcal{H}$ of the lower-associated simplicial complexes as well as  an inclusion $\Delta i:  \Delta\mathcal{A}\longrightarrow \Delta\mathcal{H}$ of the associated simplicial complexes.  We have      the following commutative diagram
\begin{eqnarray}\label{eq-zoaq}
\xymatrix{
\Delta \mathcal{A}\ar[r]^{\Delta i} &\Delta\mathcal{H}\\
\mathcal{A}\ar[r]^{i}\ar[u] & \mathcal{H}\ar[u]\\
\delta\mathcal{A}\ar[r]^{\delta i}\ar[u] &\delta\mathcal{H}. \ar[u]
}
\end{eqnarray}
Here in the diagram (\ref{eq-zoaq}),  all the vertical maps and  the middle horizontal map are injective morphisms of hypergraphs,  while the top horizontal map and the bottom horizontal map are injective simplicial maps between simplicial complexes.

Secondly,  let $C_*(\Delta\mathcal{H})$  be the chain complex  of $\Delta\mathcal{H}$  and  $C_*(\Delta\mathcal{A})$  be the chain complex  of $\Delta\mathcal{A}$.   Then    
 the  simplicial map $\Delta i$  induces a   chain map 
\begin{eqnarray}\label{eq-xzqu}
\Delta  i:   C_*(\Delta \mathcal{A})\longrightarrow  C_*(\Delta\mathcal{H}). 
\end{eqnarray}
The restrictions  of (\ref{eq-xzqu}) to the sub-chain complexes $\text{Sup}_*(\mathcal{A})$,  $\text{Inf}_*(\mathcal{A})$  and  $C_*(\delta \mathcal{A})$  of $C_*(\Delta \mathcal{A})$  give chain maps 
\begin{eqnarray*}
\text{Sup}(i):&  \text{Sup}_*(\mathcal{A})\longrightarrow \text{Sup}_*(\mathcal{H}),\\
\text{Inf}(i):&  \text{Inf}_*(\mathcal{A})\longrightarrow \text{Inf}_*(\mathcal{H}),\\
\delta  i:&   C_*(\delta \mathcal{A})\longrightarrow  C_*(\delta\mathcal{H})
\end{eqnarray*}
respectively.  
 We have a commutative diagram of chain complexes 
{\footnotesize \begin{eqnarray}
\xymatrix{
0\ar[r] &C_*(\Delta \mathcal{A})\ar[r]^{\Delta i} & C_*(\Delta\mathcal{H})\ar[r]  & C_*(\Delta\mathcal{H})/ C_*(\Delta\mathcal{A})\ar[r] &0\\
0\ar[r] &\text{Sup}_*(\mathcal{A})\ar[r]^{\text{Sup}(i)} \ar[u]& \text{Sup}_*(\mathcal{H})\ar[r] \ar[u] & \text{Sup}_*(\mathcal{H})/ \text{Sup}_*(\mathcal{A})\ar[r] \ar[u]&0\\
0\ar[r] &\text{Inf}_*(\mathcal{A})\ar[r]^{\text{Inf}(i)} \ar[u]& \text{Inf}_*(\mathcal{H})\ar[r] \ar[u] & \text{Inf}_*(\mathcal{H})/ \text{Inf}_*(\mathcal{A})\ar[r] \ar[u]^{\iota}&0\\
0\ar[r] &C_*(\delta \mathcal{A})\ar[r]^{\delta i} \ar[u]& C_*(\delta\mathcal{H})\ar[r]\ar[u]  & C_*(\delta\mathcal{H})/ C_*(\delta\mathcal{A})\ar[r] \ar[u]&0. 
}
\label{eq-2.1}
\end{eqnarray}}Here in the  diagram (\ref{eq-2.1}),  all the maps are chain maps, each row is a short exact sequence, and the first two columns are injective chain maps.

Thirdly,  by applying the homology functor  to all the chain complexes and the chain maps in the last diagram,  we have a commutative diagram of homology groups
{\footnotesize \begin{eqnarray}\label{diag-zoqaz}
\xymatrix{
H_*(\Delta \mathcal{A})\ar[r]^{(\Delta i)_*} & H_*(\Delta\mathcal{H})\ar[r]  & H_*(\Delta\mathcal{H}, \Delta\mathcal{A}) \\
H_*(\text{Sup}_*(\mathcal{A}))\ar[r]^{\text{Sup}(i)_*} \ar[u]& H_*(\text{Sup}_*(\mathcal{H}))\ar[r] \ar[u] & H_*(\text{Sup}_*(\mathcal{H})/ \text{Sup}_*(\mathcal{A}))\ar[u]\\
H_*(\text{Inf}_*(\mathcal{A}))\ar[r]^{\text{Inf}(i)_*} \ar[u]^{\cong }& H_*( \text{Inf}_*(\mathcal{H}))\ar[r] \ar[u]^{\cong} & H_*(\text{Inf}_*(\mathcal{H})/ \text{Inf}_*(\mathcal{A})) \ar[u]^{\iota_*}\\
H_*(\delta \mathcal{A})\ar[r]^{(\delta i)_*} \ar[u]& H_*(\delta\mathcal{H})\ar[r]\ar[u]  & H_*(\delta\mathcal{H}, \delta\mathcal{A}). \ar[u]
   }
   \end{eqnarray}}
Note that by \cite[Proposition~2.4 and Proposition~3.4]{hg1},  we have the  isomorphisms 
\begin{eqnarray*}
H_*(\text{Inf}_*(\mathcal{A}))\cong H_*(\text{Sup}_*(\mathcal{A})), ~~~~~~
H_*(\text{Inf}_*(\mathcal{H}))\cong H_*(\text{Sup}_*(\mathcal{H}))
\end{eqnarray*}
  in the  diagram (\ref{diag-zoqaz}).  
\begin{lemma}\label{le-2.1}
The chain map 
\begin{eqnarray*}
\iota: {\rm Inf}_*(\mathcal{H})/ {\rm Inf}_*(\mathcal{A}) \longrightarrow {\rm Sup}_*(\mathcal{H})/ {\rm Sup}_*(\mathcal{A})
\end{eqnarray*}
 in the diagram (\ref{eq-2.1}) induces an isomorphism $\iota_*$ of  homology. 
\end{lemma} 
\begin{proof}
In Lemma~\ref{le-1.1},  we substitute $C$ with the chain complex $C_*(\Delta\mathcal{H})$,  substitute $D$ with the graded abelian subgroup 
\begin{eqnarray*}
G(\mathcal{H})_*=\mathbb{Z}(\mathcal{H})_*\otimes G 
\end{eqnarray*}
of $C_*(\Delta\mathcal{H})$,  
and substitute $D'$ with the graded abelian subgroup  
\begin{eqnarray*}
G(\mathcal{A})_*=\mathbb{Z}(\mathcal{A})_*\otimes G
\end{eqnarray*}
of $C_*(\Delta\mathcal{H})$.   
 Then  Lemma~\ref{le-2.1} follows from Lemma~\ref{le-1.1}. 
\end{proof}

Finally,  we arrive at the definition of the  relative embedded homology groups  for hypergraph pairs.  
\begin{definition}\label{def-2.2}
We call the homology groups $H_*(\text{Inf}_*(\mathcal{H})/ \text{Inf}_*(\mathcal{A}) )$,   which is isomorphic to the homology groups $H_*(\text{Sup}_*(\mathcal{H})/ \text{Sup}_*(\mathcal{A}) )$,      the  {\it relative embedded homology groups} of the pair $(\mathcal{H}, \mathcal{A})$  or simply the {\it embedded homology}  of $(\mathcal{H},\mathcal{A})$.  We denote these  relative homology groups as $H_*(\mathcal{H},\mathcal{A})$.   
\end{definition}
\begin{remark}
If we allow $\mathcal{A}$ to be the empty-set $\emptyset$,  then the relative embedded homology groups $H_*(\mathcal{H},\emptyset)$ of the hypergraph pair $(\mathcal{H},\emptyset)$ is the usual embedded homology $H_*(\mathcal{H})$ of $\mathcal{H}$  (cf.  \cite[Subsection~3.2]{hg1}).  
\end{remark}

\subsection{Some examples}\label{ss3.2}

Let the abelian group $G$ be the integers $\mathbb{Z}$.    
\begin{example}
\label{ex-2.1}
Let  $\mathcal{H}=\{\{v_0\},\{v_1\},\{v_2\}, \{v_0,v_1\},\{v_0,v_1,v_2\}\}$. Then 
\begin{eqnarray*}
&{\rm Inf}_0(\mathcal{H})=\mathbb{Z}(\{v_0\},\{v_1\},\{v_2\}), ~~~
{\rm Inf}_1(\mathcal{H})=\mathbb{Z}(\{v_0,v_1\}), ~~~ 
{\rm Inf}_2(\mathcal{H})=0,\\ 
&{\rm Sup}_0(\mathcal{H})=\mathbb{Z}(\{v_0\},\{v_1\},\{v_2\}),  ~~~ {\rm Sup}_2(\mathcal{H})= \mathbb{Z}(\{v_0,v_1,v_2\}),  
\\
&{\rm Sup}_1(\mathcal{H})=\mathbb{Z}(\{v_0,v_1\},\{v_1,v_2\}-\{v_0,v_2\}+\{v_0,v_1\})
\\
&=\mathbb{Z}(\{v_0,v_1\},\{v_1,v_2\}-\{v_0,v_2\}). 
\end{eqnarray*}

\noindent (i). Let $\mathcal{A}=\{\{v_0\},\{v_1\},\{v_0,v_1\}\}$.  
Since $\mathcal{A}$  is a simplicial complex,  we  have 
\begin{eqnarray*}
&{\rm Inf}_0(\mathcal{A})={\rm Sup}_0(\mathcal{A})=\mathbb{Z}(\{v_0\},\{v_1\}),~~~
{\rm Inf}_1(\mathcal{A})={\rm Sup}_1(\mathcal{A})=\mathbb{Z}(\{v_0,v_1\}),\\
&{\rm Inf}_2(\mathcal{A})={\rm Sup}_2(\mathcal{A})=0. 
\end{eqnarray*}
Hence
\begin{eqnarray*}
&{\rm Inf}_0(\mathcal{H})/{\rm Inf}_0(\mathcal{A})=\mathbb{Z}(\{v_2\}),~~~ 
{\rm Inf}_1(\mathcal{H})/{\rm Inf}_1(\mathcal{A})={\rm Inf}_2(\mathcal{H})/{\rm Inf}_2(\mathcal{A})
=0,\\
&{\rm Sup}_0(\mathcal{H})/{\rm Sup}_0(\mathcal{A})=\mathbb{Z}(\{v_2\}),~~~ 
{\rm Sup}_1(\mathcal{H})/{\rm Sup}_1(\mathcal{A})=\mathbb{Z}(\{v_1,v_2\}-\{v_0,v_2\}),  
\\&{\rm Sup}_2(\mathcal{H})/{\rm Sup}_2(\mathcal{A})
=\mathbb{Z}(\{v_0,v_1,v_2\}) 
\end{eqnarray*}
where  the boundary maps of the  quotient chain complexes  are given by 
\begin{eqnarray*}
&\partial_0(\{v_2\})=0, ~~~ \partial_1(\{v_1,v_2\}-\{v_0,v_2\})=0,   \\
&\partial_2(\{v_0,v_1,v_2\})=\{v_1,v_2\}-\{v_0,v_2\}. 
\end{eqnarray*}
Therefore, 
$H_0(\mathcal{H},\mathcal{A})=\mathbb{Z}$  and  $H_1(\mathcal{H},\mathcal{A})=H_2(\mathcal{H},\mathcal{A})=0$.

\noindent  (ii). Let $\mathcal{A}'=\{\{v_0\},\{v_2\},\{v_0,v_1\}\}$.  Then 
\begin{eqnarray*}
&{\rm Inf}_0(\mathcal{A}')=\mathbb{Z}(\{v_0\},\{v_2\}),~~~
{\rm Inf}_1(\mathcal{A}')=
{\rm Inf}_2(\mathcal{A}')={\rm Sup}_2(\mathcal{A}')=0,\\
& {\rm Sup}_0(\mathcal{A}')=\mathbb{Z}(\{v_0\},\{v_2\},\{v_1\}-\{v_0\})=\mathbb{Z}(\{v_0\},\{v_1\},\{v_2\}), \\
& {\rm Sup}_1(\mathcal{A}')=\mathbb{Z}(\{v_0,v_1\}). 
\end{eqnarray*}
Hence 
\begin{eqnarray*}
&{\rm Inf}_0(\mathcal{H})/{\rm Inf}_0(\mathcal{A}')=\mathbb{Z}(\{v_1\}),~~~ 
{\rm Inf}_1(\mathcal{H})/{\rm Inf}_1(\mathcal{A}')=\mathbb{Z}(\{v_0,v_1\}), \\
& {\rm Inf}_2(\mathcal{H})/{\rm Inf}_2(\mathcal{A}')
={\rm Sup}_0(\mathcal{H})/{\rm Sup}_0(\mathcal{A}')=0,\\
&{\rm Sup}_1(\mathcal{H})/{\rm Sup}_1(\mathcal{A}')=\mathbb{Z}(\{v_1,v_2\}-\{v_0,v_2\}),  
\\&{\rm Sup}_2(\mathcal{H})/{\rm Sup}_2(\mathcal{A}')
=\mathbb{Z}(\{v_0,v_1,v_2\})  
\end{eqnarray*}
where  the boundary maps of the  quotient chain complexes  are given by 
\begin{eqnarray*}
&\partial_0(\{v_1\})=\partial_1(\{v_1,v_2\}-\{v_0,v_2\})=0,~~~ \partial_1(\{v_0,v_1\})=\{v_1\},\\
& \partial_2(\{v_0,v_1,v_2\})= \{v_1,v_2\}-\{v_0,v_2\}.  
\end{eqnarray*}
Therefore,  $H_0(\mathcal{H},\mathcal{A}')=H_1(\mathcal{H},\mathcal{A}')=H_2(\mathcal{H},\mathcal{A}')=0$.

\noindent  (iii).  Let $\mathcal{A}''=\{\{v_0,v_1\},\{v_0,v_1,v_2\}\}$.  Then 
\begin{eqnarray*}
&{\rm Inf}_0(\mathcal{A}'')= {\rm Inf}_1(\mathcal{A}'')={\rm Inf}_2(\mathcal{A}'')=0,~~~{\rm Sup}_0(\mathcal{A}'')=\mathbb{Z}(\{v_1\}-\{v_0\}),\\
&  {\rm Sup}_1(\mathcal{A}'')=\mathbb{Z}(\{v_0,v_1\}, \{v_1,v_2\}-\{v_0,v_2\}), ~~~ {\rm Sup}_2(\mathcal{A}'')=\mathbb{Z}(\{v_0,v_1,v_2\}). 
\end{eqnarray*}
Hence 
\begin{eqnarray*}
&{\rm Inf}_0(\mathcal{H})/{\rm Inf}_0(\mathcal{A}'')=\mathbb{Z}(\{v_0\},\{v_1\},\{v_2\}), ~~~
{\rm Inf}_1(\mathcal{H})/{\rm Inf}_1(\mathcal{A}'')=\mathbb{Z}(\{v_0,v_1\}),  
\\
&{\rm Inf}_2(\mathcal{H})/{\rm Inf}_2(\mathcal{A}'')={\rm Sup}_1(\mathcal{H})/{\rm  Sup}_1(\mathcal{A}'')={\rm Sup}_2(\mathcal{H})/{\rm  Sup}_2(\mathcal{A}'')=0,\\
&{\rm Sup}_0(\mathcal{H})/{\rm  Sup}_0(\mathcal{A}'')=\mathbb{Z}(\{v_0\}, \{v_2\})   \end{eqnarray*}
where  the boundary maps of the  quotient chain complexes  are given by 
\begin{eqnarray*}
\partial_0(\{v_0\})=\partial_0(\{v_1\})=\partial_0(\{v_2\})=0, ~~~ \partial_1(\{v_0,v_1\})=\{v_1\}-\{v_0\}. 
\end{eqnarray*}
Therefore, 
$H_0(\mathcal{H},\mathcal{A}'')=\mathbb{Z}\oplus\mathbb{Z}$  and  $H_1(\mathcal{H},\mathcal{A}'')=H_2(\mathcal{H},\mathcal{A}'')=0$.    
\end{example}

The relative embedded homology groups of the hypergraph pairs (i), (ii), (iii) of  Example~\ref{ex-2.1} are summarized in  the  next  table. 
{\footnotesize  
\begin{table}[ht]
\caption{Example~\ref{ex-2.1}}  
\centering  
\begin{tabular}{|c |c | c c c c |}  
\hline  
 name  & notation   & $n=0$ & $n=1$  & $n=2$ & $n\geq 3$  \\
\hline
&${\rm Inf}_n(\mathcal{H})/{\rm Inf}_n(\mathcal{A})$   &$\mathbb{Z}$& $0$ &$0$ &$0$\\
{\rm  infimum~chain~complex}&${\rm  Inf}_n(\mathcal{H})/{\rm Inf}_n(\mathcal{A}')$   &$\mathbb{Z}$ & $\mathbb{Z}$   &$0$ &$0$ \\
&${\rm  Inf}_n(\mathcal{H})/{\rm Inf}_n(\mathcal{A}'')$   & $\mathbb{Z}^{\oplus 3}$ & $\mathbb{Z}$ & $0$  &$0$ \\
\hline
&${\rm Sup}_n(\mathcal{H})/{\rm Sup}_n(\mathcal{A})$   &$\mathbb{Z}$&$\mathbb{Z}$&$\mathbb{Z}$&$0$\\
{\rm   supremum~chain~complex}&${\rm  Sup}_n(\mathcal{H})/{\rm Sup}_n(\mathcal{A}')$   &$0$ &$\mathbb{Z}$&$\mathbb{Z}$&$0$\\
&${\rm  Sup}_n(\mathcal{H})/{\rm Sup}_n(\mathcal{A}'')$   &$\mathbb{Z}\oplus\mathbb{Z}$&$0$&$0$&$0$\\
\hline
&  $H_n(\mathcal{H},\mathcal{A})$ & $\mathbb{Z}$ & $0$ & $0$ & $0$  \\
{\rm  embedded homology}& $H_n(\mathcal{H},\mathcal{A}')$ & $0$ & $0$ & $0$ & $0$  \\
&  $H_n(\mathcal{H},\mathcal{A}'')$ & $\mathbb{Z}\oplus\mathbb{Z}$ & $0$ & $0$ & $0$ \\  
\hline  
\end{tabular}
\label{x}  
\end{table}}The geometric realizations of the hypergraphs $\mathcal{H}$, $\mathcal{A}$,  $\mathcal{A}'$ and $\mathcal{A}''$  in (i), (ii), (iii)  of Example~\ref{ex-2.1}  are illustrated   in  Figure~1. 
\begin{figure}[!htbp] 
\label{fig-1}
 \begin{center}
\begin{tikzpicture}
 \coordinate[label=left:$\mathcal{H}$:] (G) at (0.5,2.5);
\coordinate      (A) at (1,1); 
 \coordinate    (B) at (4,1); 
 \coordinate      (C) at (3,3);
 \coordinate[label=left:$v_0$] (A) at (1,1);
 \coordinate[label=right:$v_1$] (B) at (4,1);
 \coordinate[label=right:$v_2$] (C) at (3,3);
\draw (1,1) circle (0.08); 
\draw (4,1) circle (0.08); 
\draw (3,3) circle (0.08); 
\fill (1,1) circle (2.2pt);
\fill (4,1) circle (2.2pt);
\fill (3,3) circle (2.2pt);
\draw [dotted] (B) -- (C);
\draw [dotted] (A) -- (C);
\draw [line width=1pt] (A) -- (B);
\fill [fill opacity=0.3][gray!100!white] (A) -- (B) -- (C) -- cycle;
\coordinate[label=left:$\mathcal{A}$:] (H) at (6.5,2.5);
\coordinate      (D) at (7,1); 
 \coordinate    (E) at (10,1); 
 \coordinate[label=left:$v_0$] (D) at (7,1);
 \coordinate[label=right:$v_1$] (E) at (10,1);
\draw (7,1) circle (0.08); 
\draw (10,1) circle (0.08); 
\fill (7,1) circle (2.2pt);
\fill (10,1) circle (2.2pt);
\draw [line width=1pt] (D) -- (E);
\coordinate[label=left:$\mathcal{A}'$:] (H') at (0.5,6);
\coordinate      (A') at (1,4.5); 
 \coordinate    (B') at (4,4.5); 
 \coordinate      (C') at (3,6.5);
 \coordinate[label=left:$v_0$] (A') at (1,4.5);
 \coordinate[label=right:$v_1$] (B') at (4,4.5);
 \coordinate[label=right:$v_2$] (C') at (3,6.5);
\draw (1,4.5) circle (0.08); 
\draw (4,4.5) circle (0.08); 
\draw (3,6.5) circle (0.08); 
\fill (1,4.5) circle (2.2pt);
\fill (3,6.5) circle (2.2pt);
\draw [line width=1pt] (A') -- (B');
\coordinate[label=left:$\mathcal{A}''$:] (H'') at (6.5,6);
\coordinate      (A'') at (7,4.5); 
 \coordinate    (B'') at (10,4.5); 
 \coordinate      (C'') at (9,6.5);
 \coordinate[label=left:$v_0$] (A'') at (7,4.5);
 \coordinate[label=right:$v_1$] (B'') at (10,4.5);
 \coordinate[label=right:$v_2$] (C'') at (9,6.5);
\draw (7,4.5) circle (0.08); 
\draw (10,4.5) circle (0.08); 
\draw (9,6.5) circle (0.08); 
\draw [dotted] (B'') -- (C'');
\draw [dotted] (A'') -- (C'');
\draw [line width=1pt] (A'') -- (B'');
\fill [fill opacity=0.3][gray!100!white] (A'') -- (B'') -- (C'') -- cycle;
\end{tikzpicture}
   
  \caption{Example~\ref{ex-2.1}}  
\end{center}
\end{figure}

\begin{example}\label{ex-2.2}
Let $\mathcal{A}=\{\{v_0\},\{v_1\},\{v_0,v_1\},\{v_0,v_1,v_2\}\}$.  Then 
\begin{eqnarray*}
&{\rm Inf}_0(\mathcal{A})={\rm Sup}_0(\mathcal{A})=\mathbb{Z}(\{v_0\},\{v_1\}),~~~
{\rm Inf}_1(\mathcal{A})=\mathbb{Z}(\{v_0, v_1\}), ~~~
{\rm Inf}_2(\mathcal{A})=0,\\
& 
{\rm Sup}_1(\mathcal{A})=\mathbb{Z}(\{v_0, v_1\}, \{v_1,v_2\}-\{v_0,v_2\}), ~~~
{\rm Sup}_2(\mathcal{A})=\mathbb{Z}(\{v_0,v_1,v_2\}). 
\end{eqnarray*}

\noindent  (i). Let $
\mathcal{H}= \{\{v_0\},\{v_1\},\{v_0,v_1\}, \{v_1,v_2\}, \{v_0,v_2\},\{v_0,v_1,v_2\}\}$. 
Then 
\begin{eqnarray*}
&{\rm Inf}_0(\mathcal{H})=\mathbb{Z}(\{v_0\},\{v_1\}),~~~
{\rm Inf}_1(\mathcal{H})=\mathbb{Z}(\{v_0, v_1\},\{v_0,v_2\}-\{v_1,v_2\}),\\
&{\rm Inf}_2(\mathcal{H})={\rm Sup}_2(\mathcal{H})=\mathbb{Z}(\{v_0,v_1,v_2\}),\\ 
&{\rm Sup}_0(\mathcal{H})=\mathbb{Z}(\{v_0\},\{v_1\},\{v_2\}),~~~
{\rm Sup}_1(\mathcal{H})=\mathbb{Z}(\{v_0, v_1\},\{v_0,v_2\},\{v_1,v_2\}).    
\end{eqnarray*}
Hence 
\begin{eqnarray*}
&{\rm Inf}_0(\mathcal{H})/{\rm Inf}_0(\mathcal{A})={\rm Sup}_2(\mathcal{H})/{\rm Sup}_2(\mathcal{A})=0,\\
&{\rm Inf}_1(\mathcal{H})/{\rm Inf}_1(\mathcal{A})=\mathbb{Z}(\{v_1,v_2\}-\{v_0,v_2\}),\\
&{\rm Inf}_2(\mathcal{H})/{\rm Inf}_2(\mathcal{A})=\mathbb{Z}(\{v_0,v_1,v_2\}), \\
&  {\rm Sup}_0(\mathcal{H})/{\rm Sup}_0(\mathcal{A})=\mathbb{Z}(\{v_2\}),\\
&{\rm Sup}_1(\mathcal{H})/{\rm Sup}_1(\mathcal{A})=\mathbb{Z}(\{v_0,v_2\}). 
\end{eqnarray*}
where  the boundary maps of the  quotient chain complexes  are given by 
\begin{eqnarray*}
&\partial_0(\{v_2\})=\partial_1(\{v_1,v_2\}-\{v_0,v_2\})=0,~~~ \partial_1(\{v_0,v_2\})=\{v_2\}, \\
& \partial_2(\{v_0,v_1,v_2\})= \{v_1,v_2\}-\{v_0,v_2\}.  
\end{eqnarray*}
Therefore,
$
H_0(\mathcal{H},\mathcal{A})=H_1(\mathcal{H},\mathcal{A})=H_2(\mathcal{H},\mathcal{A})=0$.

\noindent  (ii).  Let $
\mathcal{H}'= \{\{v_0\},\{v_1\},\{v_2\}, \{v_0,v_1\}, \{v_1,v_2\}, \{v_0,v_2\},\{v_0,v_1,v_2\}\}$.   Since $\mathcal{H}'$  is a simplicial complex,  we  have 
\begin{eqnarray*}
&{\rm Inf}_0(\mathcal{H}')={\rm Sup}_0(\mathcal{H}')=\mathbb{Z}(\{v_0\},\{v_1\},\{v_2\}),\\
&{\rm Inf}_1(\mathcal{H}')={\rm Sup}_1(\mathcal{H}')=\mathbb{Z}(\{v_0, v_1\},\{v_1,v_2\},\{v_0,v_2\}),\\
&{\rm Inf}_2(\mathcal{H}')={\rm Sup}_2(\mathcal{H}')=\mathbb{Z}(\{v_0,v_1,v_2\}). 
\end{eqnarray*}
Hence 
\begin{eqnarray*}
&{\rm Inf}_0(\mathcal{H}')/{\rm Inf}_0(\mathcal{A})={\rm  Sup}_0(\mathcal{H}')/{\rm  Sup}_0(\mathcal{A})=\mathbb{Z}(\{v_2\}),\\
&{\rm Inf}_1(\mathcal{H}')/{\rm Inf}_1(\mathcal{A})
=\mathbb{Z}(\{v_1,v_2\},\{v_0,v_2\}),\\
&{\rm Inf}_2(\mathcal{H}')/{\rm Inf}_2(\mathcal{A})=\mathbb{Z}(\{v_0,v_1,v_2\}), \\
&  {\rm  Sup}_1(\mathcal{H}')/{\rm  Sup}_1(\mathcal{A})=\mathbb{Z}(\{v_0,v_2\}),~~~
{\rm Sup}_2(\mathcal{H}')/{\rm Sup}_2(\mathcal{A})=0
\end{eqnarray*}
where  the boundary maps of the  quotient chain complexes  are given by 
\begin{eqnarray*}
&\partial_0(\{v_2\})=0,~~~ \partial_1(\{v_0,v_2\})=\partial_1(\{v_1,v_2\})=\{v_2\},\\
&\partial_2(\{v_0,v_1,v_2\})= \{v_1,v_2\}-\{v_0,v_2\}.  
\end{eqnarray*}
Therefore,
$
H_0(\mathcal{H}',\mathcal{A})=\mathbb{Z}$  and  $H_1(\mathcal{H}',\mathcal{A})= H_2(\mathcal{H}',\mathcal{A})=0$.

\noindent  (iii).  Let $
\mathcal{H}''= \{\{v_0\},\{v_1\},\{v_2\}, \{v_0,v_1\}, \{v_0,v_1,v_2\}\}$. 
Then 
\begin{eqnarray*}
&{\rm Inf}_0(\mathcal{H}'') ={\rm Sup}_0(\mathcal{H}'') = \mathbb{Z}(\{v_0\},\{v_1\},\{v_2\}), \\
&{\rm Inf}_1(\mathcal{H}'') = \mathbb{Z}(\{v_0, v_1\}), ~~~
{\rm Inf}_2(\mathcal{H}'') = 0, \\
&
{\rm Sup}_1(\mathcal{H}'') = \mathbb{Z}(\{v_0, v_1\},\{v_1,v_2\}-\{v_0,v_2\}), ~~~
{\rm Sup}_2(\mathcal{H}'') = \mathbb{Z}(\{v_0,v_1,v_2\}).   
\end{eqnarray*}
Hence 
\begin{eqnarray*}
&{\rm Inf}_0(\mathcal{H}'')/{\rm Inf}_0(\mathcal{A})={\rm Sup}_0(\mathcal{H}'')/{\rm Sup}_0(\mathcal{A})=\mathbb{Z}(\{v_2\}),\\
&{\rm Inf}_1(\mathcal{H}'')/{\rm Inf}_1(\mathcal{A})=
{\rm Inf}_2(\mathcal{H}'')/{\rm Inf}_2(\mathcal{A})={\rm Sup}_1(\mathcal{H}'')/{\rm Sup}_1(\mathcal{A})\\
&={\rm Sup}_2(\mathcal{H}'')/{\rm Sup}_2(\mathcal{A})=0
\end{eqnarray*}
where  the boundary maps of the  quotient chain complexes  are  the zero maps.   
Therefore,
$
H_0(\mathcal{H}'',\mathcal{A})=\mathbb{Z}$  and  $H_1(\mathcal{H}'',\mathcal{A})= H_2(\mathcal{H}'',\mathcal{A})=0$.   
\end{example}

The relative embedded homology groups of the hypergraph pairs (i), (ii), (iii)  of  Example~\ref{ex-2.2}  are summarized in the  next  table.  
{\footnotesize \begin{table}[ht]
\caption{Example~\ref{ex-2.2}}  
\centering  
\begin{tabular}{|c |c | c c c c |}  
\hline  
 name  & notation   & $n=0$ & $n=1$  & $n=2$ & $n\geq 3$  \\
\hline
&${\rm Inf}_n(\mathcal{H})/{\rm Inf}_n(\mathcal{A})$   &$0$& $\mathbb{Z}$ &$\mathbb{Z}$ &$0$\\
{\rm  infimum~chain~complex}&${\rm  Inf}_n(\mathcal{H}')/{\rm Inf}_n(\mathcal{A})$   &$\mathbb{Z}$ & $\mathbb{Z}\oplus\mathbb{Z}$   &$\mathbb{Z}$ &$0$ \\
&${\rm  Inf}_n(\mathcal{H}'')/{\rm Inf}_n(\mathcal{A})$   & $\mathbb{Z}$ & $0$ & $0$  &$0$ \\
\hline
&${\rm Sup}_n(\mathcal{H})/{\rm Sup}_n(\mathcal{A})$   &$\mathbb{Z}$&$\mathbb{Z}$&$0$&$0$\\
{\rm   supremum~chain~complex}&${\rm  Sup}_n(\mathcal{H}')/{\rm Sup}_n(\mathcal{A})$   &$\mathbb{Z}$ &$\mathbb{Z}$&$0$&$0$\\
&${\rm  Sup}_n(\mathcal{H}'')/{\rm Sup}_n(\mathcal{A})$   &$\mathbb{Z}$&$0$&$0$&$0$\\
\hline
&  $H_n(\mathcal{H},\mathcal{A})$ & $0$ & $0$ & $0$ & $0$  \\
{\rm  embedded homology}& $H_n(\mathcal{H}',\mathcal{A})$ & $\mathbb{Z}$ & $0$ & $0$ & $0$  \\
&  $H_n(\mathcal{H}'',\mathcal{A})$ & $\mathbb{Z}$ & $0$ & $0$ & $0$ \\ 
\hline  
\end{tabular}
\label{xxx}  
\end{table}}The geometric realizations of the hypergraphs $\mathcal{A}$, $\mathcal{H}$,  $\mathcal{H}'$ and $\mathcal{H}''$  in (i), (ii), (iii)   of  Example~\ref{ex-2.2}  are illustrated   in  Figure~2. 
\begin{figure}[!htbp] 
\label{fig-2}
 \begin{center}
\begin{tikzpicture}
 \coordinate[label=left:$\mathcal{A}$:] (G) at (0.5,2.5);
\coordinate      (A) at (1,1); 
 \coordinate    (B) at (4,1); 
 \coordinate      (C) at (3,3);
 \coordinate[label=left:$v_0$] (A) at (1,1);
 \coordinate[label=right:$v_1$] (B) at (4,1);
 \coordinate[label=right:$v_2$] (C) at (3,3);
\draw (1,1) circle (0.08); 
\draw (4,1) circle (0.08); 
\draw (3,3) circle (0.08); 
\fill (1,1) circle (2.2pt);
\fill (4,1) circle (2.2pt);
\draw [dotted] (B) -- (C);
\draw [dotted] (A) -- (C);
\draw [line width=1pt] (A) -- (B);
\fill [fill opacity=0.3][gray!100!white] (A) -- (B) -- (C) -- cycle;
\coordinate[label=left:$\mathcal{H}$:] (H) at (6.5,2.5);
\coordinate      (D) at (7,1); 
 \coordinate    (E) at (10,1); 
 \coordinate[label=right:$v_2$] (F) at (9,3);
 \coordinate[label=left:$v_0$] (D) at (7,1);
 \coordinate[label=right:$v_1$] (E) at (10,1);
 \coordinate[label=right:$v_2$] (F) at (9,3);
\draw (7,1) circle (0.08); 
\draw (10,1) circle (0.08); 
\draw (9,3) circle (0.08); 
\fill (7,1) circle (2.2pt);
\fill (10,1) circle (2.2pt);
\draw [line width=1pt] (D) -- (E);
\draw [line width=1pt] (D) -- (F);
\draw [line width=1pt] (F) -- (E);
\fill [fill opacity=0.3][gray!100!white] (D) -- (E) -- (F) -- cycle;
\coordinate[label=left:$\mathcal{H}'$:] (H') at (0.5,6);
\coordinate      (A') at (1,4.5); 
 \coordinate    (B') at (4,4.5); 
 \coordinate      (C') at (3,6.5);
 \coordinate[label=left:$v_0$] (A') at (1,4.5);
 \coordinate[label=right:$v_1$] (B') at (4,4.5);
 \coordinate[label=right:$v_2$] (C') at (3,6.5);
\draw (1,4.5) circle (0.08); 
\draw (4,4.5) circle (0.08); 
\draw (3,6.5) circle (0.08); 
\fill (1,4.5) circle (2.2pt);
\fill (3,6.5) circle (2.2pt);
\fill (4,4.5) circle (2.2pt);
\draw [line width=1pt] (A') -- (B');
\draw [line width=1pt] (B') -- (C');
\draw [line width=1pt] (A') -- (C');
\fill [fill opacity=0.3][gray!100!white] (A') -- (B') -- (C') -- cycle;
\coordinate[label=left:$\mathcal{H}''$:] (H'') at (6.5,6);
\coordinate      (A'') at (7,4.5); 
 \coordinate    (B'') at (10,4.5); 
 \coordinate      (C'') at (9,6.5);
 \coordinate[label=left:$v_0$] (A'') at (7,4.5);
 \coordinate[label=right:$v_1$] (B'') at (10,4.5);
 \coordinate[label=right:$v_2$] (C'') at (9,6.5);
\draw (7,4.5) circle (0.08); 
\draw (10,4.5) circle (0.08); 
\draw (9,6.5) circle (0.08); 
\fill (7,4.5) circle (2.2pt); 
\fill (10,4.5) circle (2.2pt); 
\fill (9,6.5) circle (2.2pt); 
\draw [dotted] (B'') -- (C'');
\draw [dotted] (A'') -- (C'');
\draw [line width=1pt] (A'') -- (B'');
\fill [fill opacity=0.3][gray!100!white] (A'') -- (B'') -- (C'') -- cycle;
\end{tikzpicture}
  \caption{Example~\ref{ex-2.2}}  
\end{center}
\end{figure}

For any $n\geq 1$,  let $[n]$ be the hypergraph consisting of a single hyperedge $\{v_0,v_1,\cdots,v_n\}$ and  let $\Delta[n]$ be the associated simplicial complex  of $[n]$.  Note that the lower-associated simplicial complex $\delta[n]$ of $[n]$  is the empty-set.   
\begin{example}\label{ex-2.3}
Let $\mathcal{H}=\{\sigma\in \Delta[3]\mid \dim\sigma\leq 2\}$.  Then $\mathcal{H}$ is a simplicial complex.  Thus 
\begin{eqnarray*}
&{\rm Inf}_0(\mathcal{H})={\rm Sup}_0(\mathcal{H})=\mathbb{Z}(\{v_0\},\{v_1\},\{v_2\},\{v_3\}),\\
&{\rm Inf}_1(\mathcal{H}) = {\rm Sup}_1(\mathcal{H})=\mathbb{Z}(\{v_0,v_1\},\{v_0,v_2\},\{v_0,v_3\},\{v_1,v_2\},\{v_1,v_3\},\{v_2,v_3\}),\\
&{\rm Inf}_2(\mathcal{H})={\rm Sup}_2(\mathcal{H})=\mathbb{Z}(\{v_0,v_1,v_2\},\{v_0,v_1,v_3\},\{v_0,v_2,v_3\},\{v_1,v_2,v_3\}). 
\end{eqnarray*}

\noindent  (i).  Let $\mathcal{A}=\{\sigma\in\mathcal{H}\mid   1\leq \dim\sigma\leq 2\}$.  Then 
\begin{eqnarray*}
&{\rm Inf}_1(\mathcal{A})=\mathbb{Z}(\{v_1,v_2\}-\{v_0,v_2\}+\{v_0,v_1\},\{v_1,v_3\}-\{v_0,v_3\}+\{v_0,v_1\},\\
& \{v_2,v_3\}-\{v_0,v_3\}+\{v_0,v_2\}),\\
& {\rm Inf}_0(\mathcal{A})=0,\\
&{\rm Inf}_2(\mathcal{A})={\rm Sup}_2(\mathcal{A})=\mathbb{Z}(\{v_0,v_1,v_2\},\{v_0,v_1,v_3\},\{v_0,v_2,v_3\},\{v_1,v_2,v_3\}),\\
& {\rm Sup}_0(\mathcal{A})=\mathbb{Z}(\{v_j\}-\{v_i\}\mid  0\leq i<j\leq 3)\\
&=\mathbb{Z}(\{v_1\}-\{v_0\}, \{v_2\}-\{v_0\}, \{v_3\}-\{v_0\}),\\
&{\rm Sup}_1(\mathcal{A})=\mathbb{Z}(\{v_0,v_1\},\{v_0,v_2\},\{v_0,v_3\},\{v_1,v_2\},\{v_1,v_3\},\{v_2,v_3\}).  
\end{eqnarray*}
Hence 
\begin{eqnarray*}
&{\rm Inf}_0(\mathcal{H})/{\rm Inf}_0(\mathcal{A}) =\mathbb{Z}(\{v_0\},\{v_1\},\{v_2\},\{v_3\}),~~~{\rm Inf}_2(\mathcal{H})/{\rm Inf}_2(\mathcal{A}) = 0\\
&{\rm Inf}_1(\mathcal{H})/{\rm Inf}_1(\mathcal{A}) =\mathbb{Z}(\{v_0,v_2\},\{v_0,v_3\},\{v_1,v_2\}) 
\end{eqnarray*}  
where the boundary maps in the quotient infimum chain complex are given by
\begin{eqnarray*}
&\partial_1(\{v_0,v_2\})=\{v_2\}-\{v_0\}, ~~~ \partial_1(\{v_0,v_3\})=\{v_3\}-\{v_0\},\\
&\partial_1(\{v_1,v_2\})=\{v_2\}-\{v_1\}, ~~~
 \partial_0(\{v_i\})=0 ~{\rm~for~}0\leq i\leq 3. 
\end{eqnarray*}
On the other hand, 
\begin{eqnarray*}
& {\rm Sup}_0(\mathcal{H})/{\rm Sup}_0(\mathcal{A})  =\mathbb{Z}(\{v_0\}), ~~~
{\rm Sup}_1(\mathcal{H})/{\rm Sup}_1(\mathcal{A})= {\rm Sup}_2(\mathcal{H})/{\rm Sup}_2(\mathcal{A})= 0
\end{eqnarray*}  
where the boundary maps in the quotient supremum chain complex are given by
$
\partial_0(\{v_0\})=0$.  Therefore,  
$H_0(\mathcal{H},\mathcal{A})=\mathbb{Z}$  and    $H_1(\mathcal{H},\mathcal{A})=H_2(\mathcal{H},\mathcal{A})=0$.

\noindent  (ii). Let  $\mathcal{A}'=\{\sigma\in\mathcal{H}\mid \dim\sigma\leq 1\}$. Then $(\mathcal{H},\mathcal{A}')$  is  a  simplicial complex pair.  Thus
\begin{eqnarray*}
&{\rm Inf}_0(\mathcal{A}')={\rm Sup}_0(\mathcal{A}')=\mathbb{Z}(\{v_0\},\{v_1\},\{v_2\},\{v_3\}),\\
&{\rm Inf}_1(\mathcal{A}') = {\rm Sup}_1(\mathcal{A}')=\mathbb{Z}(\{v_0,v_1\},\{v_0,v_2\},\{v_0,v_3\},\{v_1,v_2\},\{v_1,v_3\},\{v_2,v_3\}) 
\end{eqnarray*}
and 
\begin{eqnarray*}
&{\rm Inf}_0(\mathcal{H})/{\rm Inf}_0(\mathcal{A}')={\rm Sup}_0(\mathcal{H})/{\rm Sup}_0(\mathcal{A}')={\rm Inf}_1(\mathcal{H})/{\rm Inf}_1(\mathcal{A}')\\
&={\rm Sup}_1(\mathcal{H})/{\rm Sup}_1(\mathcal{A}')=0, \\
&{\rm Inf}_2(\mathcal{H})/{\rm Inf}_2(\mathcal{A}')={\rm Sup}_2(\mathcal{H})/{\rm Sup}_2(\mathcal{A}')\\
&=\mathbb{Z}(\{v_0,v_1,v_2\},\{v_0,v_1,v_3\},\{v_0,v_2,v_3\},\{v_1,v_2,v_3\}). 
\end{eqnarray*}
The relative embedded homology   groups  of $(\mathcal{H},\mathcal{A}')$  reduce  to the usual relative homology  groups of  simplicial complex pairs.  That is,    
$H_0(\mathcal{H},\mathcal{A}')=H_1(\mathcal{H},\mathcal{A}')=0$  and  $H_2(\mathcal{H},\mathcal{A}')=\mathbb{Z}^{\oplus 4}$.

\noindent  (iii).  Let $\mathcal{A}''=\{\sigma\in\mathcal{H}\mid \dim\sigma\neq 1\}$.  Then 
\begin{eqnarray*}
 &{\rm Inf}_0(\mathcal{A}'')={\rm Sup}_0(\mathcal{A}'')=\mathbb{Z}(\{v_0\},\{v_1\},\{v_2\},\{v_3\}),~~~{\rm Inf}_1(\mathcal{A}'')=0,\\
 &{\rm Inf}_2(\mathcal{A}'')=\mathbb{Z}(\{v_1,v_2,v_3\}-\{v_0,v_2,v_3\}+\{v_0,v_1,v_3\}-\{v_0,v_1,v_2\}),\\
&{\rm Sup}_1(\mathcal{A}'')=\mathbb{Z}(\{v_1,v_2\}-\{v_0,v_2\}+\{v_0,v_1\},\\
&\{v_1,v_3\}-\{v_0,v_3\}+\{v_0,v_1\}, \{v_2,v_3\}-\{v_0,v_3\}+\{v_0,v_2\}),\\
&{\rm Sup}_2(\mathcal{A}'')=\mathbb{Z}(\{v_0,v_1,v_2\},\{v_0,v_1,v_3\},\{v_0,v_2,v_3\},\{v_1,v_2,v_3\}). 
\end{eqnarray*} 
Hence 
\begin{eqnarray*}
&{\rm Inf}_0(\mathcal{H})/{\rm Inf}_0(\mathcal{A}'')={\rm  Sup}_0(\mathcal{H})/{\rm  Sup}_0(\mathcal{A}'')={\rm  Sup}_2(\mathcal{H})/{\rm  Sup}_2(\mathcal{A}'')=0,\\
&{\rm Inf}_1(\mathcal{H})/{\rm Inf}_1(\mathcal{A}'')=\mathbb{Z}(\{v_0,v_1\},\{v_0,v_2\},\{v_0,v_3\},\{v_1,v_2\},\{v_1,v_3\},\{v_2,v_3\}),\\
&{\rm Inf}_2(\mathcal{H})/{\rm Inf}_2(\mathcal{A}'')=\mathbb{Z}(\{v_0,v_1,v_2\},\{v_0,v_1,v_3\},\{v_0,v_2,v_3\}),\\
&{\rm Sup}_1(\mathcal{H})/{\rm Sup}_1(\mathcal{A}'')=\mathbb{Z}(\{v_0,v_2\}, \{v_0,v_3\},\{v_1, v_2\})  
\end{eqnarray*}
where the boundary maps in the quotient chain complexes are given by
\begin{eqnarray*}
&\partial_1(\{v_i,v_j\})=0 ~{\rm~for~}0\leq i<j\leq 3;\\
& \partial_2(\{v_i,v_j,k\})=\{v_j,v_k\}-\{v_i,v_k\}+\{v_i,v_j\} ~{\rm~for~}0\leq i<j<k\leq 3.  
\end{eqnarray*}
Therefore,
$H_0(\mathcal{H},\mathcal{A}'')=0$, $H_1(\mathcal{H},\mathcal{A}'')=\mathbb{Z}^{\oplus 3}$  and  $H_2(\mathcal{H},\mathcal{A}'')=0$.

\noindent  (iv). Let $\mathcal{A}'''=\{\sigma\in\mathcal{H}\mid \dim\sigma=1\}$.  Then  
\begin{eqnarray*}
&  {\rm Inf}_0(\mathcal{A}''')={\rm Inf}_2(\mathcal{A}''')={\rm Sup}_2(\mathcal{A}''')=0,\\
&  {\rm Inf}_1(\mathcal{A}''')= \mathbb{Z}(\{v_1,v_2\}-\{v_0,v_2\}+\{v_0,v_1\},\\
&\{v_1,v_3\}-\{v_0,v_3\}+\{v_0,v_1\}, \{v_2,v_3\}-\{v_0,v_3\}+\{v_0,v_2\}),\\
&{\rm Sup}_0(\mathcal{A}''')=\mathbb{Z}(\{v_j\}-\{v_i\}\mid  0\leq i<j\leq 3)\\
&=\mathbb{Z}(\{v_1\}-\{v_0\}, \{v_2\}-\{v_0\}, \{v_3\}-\{v_0\}),\\
&{\rm Sup}_1(\mathcal{A}''')=\mathbb{Z}(\{v_0,v_1\},\{v_0,v_2\},\{v_0,v_3\},\{v_1,v_2\},\{v_1,v_3\},\{v_2,v_3\}). 
\end{eqnarray*}
Hence 
\begin{eqnarray*}
&{\rm Inf}_0(\mathcal{H})/{\rm Inf}_0(\mathcal{A}''')=\mathbb{Z}(\{v_0\},\{v_1\},\{v_2\},\{v_3\}),\\
&{\rm Inf}_1(\mathcal{H})/{\rm Inf}_1(\mathcal{A}''') =\mathbb{Z}(\{v_0,v_2\},\{v_0,v_3\},\{v_1,v_2\}), \\
&{\rm Inf}_2(\mathcal{H})/{\rm Inf}_2(\mathcal{A}''')=\mathbb{Z}(\{v_0,v_1,v_2\},\{v_0,v_1,v_3\},\{v_0,v_2,v_3\},\{v_1,v_2,v_3\})
\end{eqnarray*}
where the boundary maps in the quotient infimum chain complex are given by
\begin{eqnarray*}
&\partial_0(\{v_i\})=0 ~{\rm~for~} 0\leq i\leq 3, \\
&\partial_1(\{v_0,v_2\})=\{v_2\}-\{v_0\},~~~ \partial_1(\{v_0,v_3\})=\{v_3\}-\{v_0\},~~~ \partial_1(\{v_1,v_2\})=\{v_2\}-\{v_1\},\\
&\partial_2(\{v_0,v_1,v_2\})=  \partial_2(\{v_0,v_1,v_3\})=\partial_2(\{v_0,v_2,v_3\})=\partial_2(\{v_1,v_2,v_3\})=0. 
\end{eqnarray*}
On the other hand, 
\begin{eqnarray*}
&{\rm  Sup}_0(\mathcal{H})/{\rm  Sup}_0(\mathcal{A}''')=\mathbb{Z}(\{v_0\}),~~~{\rm  Sup}_1(\mathcal{H})/{\rm  Sup}_1(\mathcal{A}''')=0,\\
&{\rm  Sup}_2(\mathcal{H})/{\rm  Sup}_2(\mathcal{A}''')=\mathbb{Z}(\{v_0,v_1,v_2\},\{v_0,v_1,v_3\},\{v_0,v_2,v_3\},\{v_1,v_2,v_3\}) 
\end{eqnarray*}
where the boundary maps in the quotient chain complexes are given by
\begin{eqnarray*}
&\partial_0(\{v_0\})=0,~~~
& \partial_2(\{v_i,v_j,k\})=0 ~{\rm~for~}0\leq i<j<k\leq 3.  
\end{eqnarray*}
Therefore,  
$
H_0(\mathcal{H},\mathcal{A}''')=\mathbb{Z}$,   $H_1(\mathcal{H},\mathcal{A}''')=0$  and  $H_2(\mathcal{H},\mathcal{A}''')=\mathbb{Z}^{\oplus 4} $.   

\end{example}

The relative embedded homology groups of the hypergraph pairs (i), (ii), (iii), (iv) of Example~\ref{ex-2.3} are summarized in  the  next  table. 

\begin{table}[ht]

\caption{Example~\ref{ex-2.3}}  
\centering  
\begin{tabular}{|c |c | c c c c |}  
\hline  
 name  & notation   & $n=0$ & $n=1$  & $n=2$ & $n\geq 3$  \\
\hline
&${\rm Inf}_n(\mathcal{H})/{\rm Inf}_n(\mathcal{A})$   &$\mathbb{Z}^{\oplus 4}$& $\mathbb{Z}^{\oplus 3}$ &$0$ &$0$\\
{\rm  infimum~chain~complex}&${\rm  Inf}_n(\mathcal{H})/{\rm Inf}_n(\mathcal{A}')$   &$0$ & $0$   &$\mathbb{Z}^{\oplus 4}$ &$0$ \\
&${\rm  Inf}_n(\mathcal{H})/{\rm Inf}_n(\mathcal{A}'')$   & $0$ & $\mathbb{Z}^{\oplus 6}$ & $\mathbb{Z}^{\oplus 3}$  &$0$ \\
&${\rm  Inf}_n(\mathcal{H})/{\rm Inf}_n(\mathcal{A}''')$   & $\mathbb{Z}^{\oplus 4}$ & $\mathbb{Z}^{\oplus 3}$ & $\mathbb{Z}^{\oplus 4}$  &$0$ \\
\hline
&${\rm Sup}_n(\mathcal{H})/{\rm Sup}_n(\mathcal{A})$   &$\mathbb{Z}$&$0$&$0$&$0$\\
{\rm   supremum~chain~complex}&${\rm  Sup}_n(\mathcal{H})/{\rm Sup}_n(\mathcal{A}')$   &$0$ & $0$   &$\mathbb{Z}^{\oplus 4}$ &$0$\\
&${\rm  Sup}_n(\mathcal{H})/{\rm Sup}_n(\mathcal{A}'')$   &$0$&$\mathbb{Z}^{\oplus 3}$&$0$&$0$\\
&${\rm  Sup}_n(\mathcal{H})/{\rm Sup}_n(\mathcal{A}''')$   &$\mathbb{Z}$&$0$&$\mathbb{Z}^{\oplus 4}$&$0$\\
\hline
&  $H_n(\mathcal{H},\mathcal{A})$ & $\mathbb{Z}$ & $0$ & $0$ & $0$  \\
{\rm  embedded homology}& $H_n(\mathcal{H},\mathcal{A}')$ & $0$ & $0$   &$\mathbb{Z}^{\oplus 4}$ & $0$  \\
&  $H_n(\mathcal{H},\mathcal{A}'')$ & $0$ & $\mathbb{Z}^{\oplus  3}$ & $0$ & $0$ \\
&  $H_n(\mathcal{H},\mathcal{A}''')$ & $\mathbb{Z}$ & $0$ & $\mathbb{Z}^{\oplus 4}$ & $0$ \\  
\hline  
\end{tabular}
\label{xxxxxx}  
\end{table}

\section{Morphisms of hypergraph  pairs  and  homomorphisms of  the  homology}\label{s4x}

In  this section,  we  define the morphisms of hypergraph pairs and  prove that such a morphism induces   homomorphisms  of the relative embedded homology groups.  Moreover,  we  give  some long exact sequences  of homology groups   for hypergraph pairs.

\subsection{Morphisms of  hypergraph  pairs and  the  induced  homomorphisms of  the  embedded  homology}\label{ss4.1}

Let $V$  and $V'$  be  discrete sets.  
Let $\mathcal{H}$ be a hypergraph with vertices from $V$  and $\mathcal{H}'$ be a hypergraph with vertices from $V'$.   Let $\mathcal{A}$ be a sub-hypergraph of $\mathcal{H}$ and $\mathcal{A}'$  be a sub-hypergraph of $\mathcal{H}'$.    By Definition~\ref{def-cmahg},  we  have hypergraph pairs $(\mathcal{H},\mathcal{A})$  and  $(\mathcal{H}',\mathcal{A}')$.  

\begin{definition}
A {\it morphism  $f: (\mathcal{H},\mathcal{A})\longrightarrow (\mathcal{H}',\mathcal{A}')$ of hypergraph pairs} is a morphism $f: \mathcal{H}\longrightarrow \mathcal{H}'$ of hypergraphs such that    
   for any hyperedge $\sigma$ of  $\mathcal{A}$,  its  image $f(\sigma)$ is a hyperedge of $\mathcal{A}'$. 
\end{definition}

It can be verified that  all  the  chain  maps  in  the commutative diagram (\ref{eq-2.1})  are functorial with respect to morphisms of hypergraph pairs.  Thus  by   generalizing \cite[Proposition~3.7]{hg1},   it  follows    in the next lemma   that  the relative embedded homology  groups  are    functorial   as  well.   

\begin{lemma}
A morphism  $f: (\mathcal{H},\mathcal{A})\longrightarrow (\mathcal{H}',\mathcal{A}')$ of hypergraph pairs induces a homomorphism $f_*: H_*(\mathcal{H},\mathcal{A})\longrightarrow H_*(\mathcal{H}',\mathcal{A}')$  of relative embedded homology  groups.   
\end{lemma}
\begin{proof}
 Let  $f: (\mathcal{H},\mathcal{A})\longrightarrow (\mathcal{H}',\mathcal{A}')$   be  a  morphism of hypergraph pairs.   Then  we  have a commutative diagram of chain complexes 
 {\footnotesize \begin{eqnarray*}
 \xymatrix{
 C_*(\Delta\mathcal{H})/C_*(\Delta\mathcal{A}) \ar[r] \ar[d] & {\rm Sup}_*(\mathcal{H})/{\rm Sup}_*(\mathcal{A}) \ar[r]  \ar[d] & {\rm Inf}_*(\mathcal{H})/{\rm Inf}_*(\mathcal{A}) \ar[r] \ar[d] & C_*(\delta\mathcal{H})/C_*(\delta\mathcal{A})  \ar[d] \\
C_*(\Delta\mathcal{H}')/C_*(\Delta\mathcal{A}') \ar[r]   & {\rm Sup}_*(\mathcal{H}')/{\rm Sup}_*(\mathcal{A}') \ar[r]    & {\rm Inf}_*(\mathcal{H}')/{\rm Inf}_*(\mathcal{A}') \ar[r]   & C_*(\delta\mathcal{H}')/C_*(\delta\mathcal{A}').       
 }
 \end{eqnarray*}}In the last diagram, all the maps are chain  maps.  Applying the homology functor,  for each $n\geq  0$  we  have a commutative diagram of homology groups 
{\footnotesize \begin{eqnarray*} 
 \xymatrix{
 H_n(\Delta\mathcal{H},\Delta\mathcal{A}) \ar[r] \ar[d] & H_n({\rm Sup}_*(\mathcal{H})/{\rm Sup}_*(\mathcal{A})) \ar[r]^{\cong }  \ar[d] & H_n({\rm Inf}_*(\mathcal{H})/{\rm Inf}_*(\mathcal{A})) \ar[r] \ar[d] & H_n (\delta\mathcal{H},\delta\mathcal{A})  \ar[d] \\
H_n( \Delta\mathcal{H}',\Delta\mathcal{A}') \ar[r]   & H_n({\rm Sup}_*(\mathcal{H}')/{\rm Sup}_*(\mathcal{A}')) \ar[r] ^{\cong}   & H_n( {\rm Inf}_*(\mathcal{H}')/{\rm Inf}_*(\mathcal{A}') )\ar[r]   & H_n(\delta\mathcal{H}',\delta\mathcal{A}').       
 }
 \end{eqnarray*}}In the last diagram,  the second map in the first row as well as the second map in the second row is  an isomorphism.  Thus the second vertical map and the third  vertical map can  be  identified.    We  denote this map as  $f_*$.   By  Definition~\ref{def-2.2},  for each $n\geq  0$   the map $f_*$  is a homomorphism from the relative embedded homology group $H_n(\mathcal{H},\mathcal{A})$  to the relative embedded homology group  $H_n(\mathcal{H}',\mathcal{A}')$.   
\end{proof}

The  relative embedded homology groups of hypergraph pairs satisfy  the Eilenberg-Steenrod  Axiom 3 (cf. \cite[p. 146]{eat}):  

\begin{proposition}\label{le-3.a}
Let   $f: (\mathcal{H},\mathcal{A})\longrightarrow (\mathcal{H}',\mathcal{A}')$   be   a morphism of hypergraph pairs.   Then the following diagram commutes: 
{\footnotesize\begin{eqnarray*}
\xymatrix{
H_n(\mathcal{H},\mathcal{A})\ar[r]^{f_*}\ar[d]^{\partial_*} & H_n(\mathcal{H}',\mathcal{A}')\ar[d]^{\partial_*}\\
H_{n-1}(\mathcal{A})\ar[r]^{(f\mid_{\mathcal{A}})_*} & H_{n-1}(\mathcal{A}'). }
\end{eqnarray*}}
\end{proposition}
\begin{proof}
Let $f: (\mathcal{H},\mathcal{A})\longrightarrow (\mathcal{H}',\mathcal{A}')$  be a morphism of hypergraph pairs.  Then by the functorial property of the chain maps in  the third row in  the  diagram  (\ref{eq-2.1}), we have a commutative diagram of chain complexes
{\footnotesize\begin{eqnarray*}
\xymatrix{
0\ar[r] &\text{Inf}_*(\mathcal{A})\ar[r]  \ar[d]& \text{Inf}_*(\mathcal{H})\ar[r] \ar[d] & \text{Inf}_*(\mathcal{H})/ \text{Inf}_*(\mathcal{A})\ar[r] \ar[d] &0\\
0\ar[r] &\text{Inf}_*(\mathcal{A}')\ar[r]   & \text{Inf}_*(\mathcal{H}')\ar[r]  & \text{Inf}_*(\mathcal{H}')/ \text{Inf}_*(\mathcal{A}')\ar[r]  &0.  
}
\end{eqnarray*}}Here   in the last diagram,  all the arrows are chain maps, all the vertical chain maps are induced by $f$,  and both rows are short exact sequences.  By applying the homology functor  to  the  last  diagram,   we have a commutative diagram
{\footnotesize\begin{eqnarray*}
\xymatrix{
\cdots\ar[r] & H_n(\mathcal{A})\ar[d]\ar[r]^-{ i_* } & H_n( \mathcal{H}) \ar[d]\ar[r]^-{ j_* } & H_n(\mathcal{H},\mathcal{A})\ar[d]\ar[r]^-{ \partial_* } & H_{n-1}(\mathcal{A})\ar[d]\ar[r] &\cdots\\
\cdots\ar[r] & H_n(\mathcal{A}') \ar[r]^-{ i_* } & H_n( \mathcal{H}') \ar[r]^-{ j_* } & H_n(\mathcal{H}',\mathcal{A}') \ar[r]^-{ \partial_* } & H_{n-1}(\mathcal{A}') \ar[r] &\cdots 
}
\end{eqnarray*}}where both rows are long exact sequences and  all the vertical maps are homomorphisms  induced  by $f$.  We  finish  the proof of  Proposition~\ref{le-3.a}. 
\end{proof}

\subsection{Some long exact sequences of  the  homology}

\begin{lemma}\label{le-zoey}
Let $(\mathcal{H},\mathcal{A})$  be a hypergraph pair.  Then
 we have a commutative diagram  
{\footnotesize\begin{eqnarray*} 
\xymatrix{
\cdots\ar[r] & H_n(\Delta \mathcal{A})\ar[r]^-{(\Delta i)_*} & H_n(\Delta \mathcal{H}) \ar[r]^-{(\Delta j)_*}& H_n(\Delta\mathcal{H},\Delta\mathcal{A})\ar[r]^-{\partial_*}& H_{n-1}(\Delta\mathcal{A})\ar[r] &\cdots\\
\cdots\ar[r] & H_n(\mathcal{A})\ar[u]\ar[r]^-{i_*} & H_n( \mathcal{H}) \ar[u]\ar[r]^-{j_*} & H_n(\mathcal{H},\mathcal{A})\ar[u]\ar[r]^-{\partial_*} & H_{n-1}(\mathcal{A})\ar[u]\ar[r] &\cdots\\
\cdots\ar[r] & H_n(\delta\mathcal{A})\ar[u]\ar[r]^-{(\delta i)_*} & H_n( \delta\mathcal{H}) \ar[u]\ar[r]^-{(\delta j)_*} & H_n(\delta\mathcal{H},\delta\mathcal{A})\ar[u]\ar[r]^-{\partial_*} & H_{n-1}(\delta\mathcal{A})\ar[u]\ar[r] &\cdots
}
\end{eqnarray*}}where each row is a   long exact sequence of  homology groups.  Moreover,  all the maps  in Lemma~\ref{le-zoey} are functorial  with respect to morphisms  of hypergraph pairs.
\end{lemma}

\begin{proof}
Recall that we have two induced simplicial complex pairs
$(\Delta\mathcal{H},\Delta\mathcal{A})$  and  $(\delta\mathcal{H},\delta\mathcal{A})$.  
Consider the canonical inclusions 
\begin{eqnarray*}
i:   (\mathcal{A},\emptyset)\longrightarrow (\mathcal{H}, \emptyset),
~~~~~~
j:  (\mathcal{H},\emptyset)\longrightarrow (\mathcal{H},\mathcal{A})
\end{eqnarray*}
  of hypergraph pairs.  Then we have the induced canonical inclusions  
\begin{eqnarray*}
\Delta i:   (\Delta\mathcal{A},\emptyset)\longrightarrow (\Delta\mathcal{H}, \emptyset),~~~~~~
\Delta j:   (\Delta\mathcal{H},\emptyset)\longrightarrow (\Delta\mathcal{H}, \Delta\mathcal{A})
\end{eqnarray*}
of  simplicial complex pairs  for the associated simplicial complexes and the induced canonical inclusions  
\begin{eqnarray*}
\delta i:   (\delta\mathcal{A},\emptyset)\longrightarrow (\delta\mathcal{H}, \emptyset), ~~~~~~
\delta j:  (\delta\mathcal{H},\emptyset)\longrightarrow (\delta\mathcal{H}, \delta\mathcal{A})
\end{eqnarray*}
of  simplicial complex pairs for the  lower-associated simplicial complexes.   
By applying the homology functor to the commutative diagram (\ref{eq-2.1}),  we have the   commutative diagram  Lemma~\ref{le-zoey}   
where each row is a   long exact sequence of  homology groups.  Note that  in the   diagram Lemma~\ref{le-zoey},  the top row is the  usual   long exact sequence  of   the  relative homology groups of the  simplicial complex pair  $(\Delta\mathcal{H},\Delta\mathcal{A})$  and the  bottom row  is  the  usual   long exact sequence  of   the  relative homology groups of  the  simplicial complex pair  $(\delta\mathcal{H},\delta\mathcal{A})$.

Lemma~\ref{le-3.a} says that $\partial_*$ in the commutative diagram Lemma~\ref{le-zoey} is functorial with respect to morphisms  of hypergraph pairs.   In general,  it  can proved     analogously  that   all the maps  in Lemma~\ref{le-zoey} are functorial  with respect to morphisms  of hypergraph pairs.  
\end{proof}

\begin{corollary}
Let $(\mathcal{H},\mathcal{A})$  be a hypergraph pair.  Then for each $n\geq 0$ we have 
\begin{eqnarray}\label{eq-palr}
{\rm Ker}((\partial_*)_n) = {\rm Im}((j_*)_n),  ~~~~~~  {\rm Im}((\partial_*)_n)= {\rm Ker}((i_*)_{n-1}).  
\end{eqnarray}
In addition,  if we take the abelian group $G$ as the rational  numbers $\mathbb{Q}$,  then 
\begin{eqnarray}\label{eq-kwoo}
H_n(\mathcal{H},\mathcal{A})\cong {\rm Ker}((\partial_*)_n) \oplus {\rm Im}((\partial_*)_n).  
\end{eqnarray}
\end{corollary}

\begin{proof}
The first assertion (\ref{eq-palr})  follows from the long exact sequence in the middle row of Lemma~\ref{le-zoey}. 
In addition, we take the abelian group $G$ as the rational  numbers $\mathbb{Q}$.  By the isomorphism theorem of  vector spaces, we  have the second assertion (\ref{eq-kwoo}). 
\end{proof}

\begin{definition}
A {\it hypergraph triple} $(\mathcal{H}, \mathcal{A}, \mathcal{B})$ is a triple of hypergraphs $\mathcal{H}$, $\mathcal{A}$,  and $\mathcal{B}$ such that  $\mathcal{B}\subseteq \mathcal{A}\subseteq \mathcal{H}$,  that is,   $\mathcal{A}$  is a sub-hypergraph of $\mathcal{H}$   and  $\mathcal{B}$ is a sub-hypergraph of $\mathcal{A}$. 
\end{definition}

\begin{lemma}\label{eq-2.55}
  Given a hypergraph triple $(\mathcal{H}, \mathcal{A}, \mathcal{B})$,  we  have  a commutative diagram
 {\footnotesize \begin{eqnarray*} 
\xymatrix{
\cdots\ar[r] & H_n(\Delta \mathcal{A},\Delta\mathcal{B})\ar[r]^-{ \Delta i_* } & H_n(\Delta \mathcal{H},\Delta\mathcal{B}) \ar[r]^-{ \Delta j_* }& H_n(\Delta\mathcal{H},\Delta\mathcal{A})\ar[r]^-{ \partial _* }& H_{n-1}(\Delta\mathcal{A},\Delta\mathcal{B})\ar[r] &\cdots\\
\cdots\ar[r] & H_n(\mathcal{A},\mathcal{B})\ar[u]\ar[r]^-{ i_* } & H_n( \mathcal{H},\mathcal{B}) \ar[u]\ar[r]^-{ j_* } & H_n(\mathcal{H},\mathcal{A})\ar[u]\ar[r]^-{ \partial_* } & H_{n-1}(\mathcal{A},\mathcal{B})\ar[u]\ar[r] &\cdots\\
\cdots\ar[r] & H_n(\delta\mathcal{A},\delta\mathcal{B})\ar[u]\ar[r]^-{ \delta i_* } & H_n( \delta\mathcal{H},\mathcal{B}) \ar[u]\ar[r]^-{ \delta j_* } & H_n(\delta\mathcal{H},\delta\mathcal{A})\ar[u]\ar[r]^-{ \partial _* } & H_{n-1}(\delta\mathcal{A},\mathcal{B})\ar[u]\ar[r] &\cdots}
\end{eqnarray*}}where each row is a long exact sequence of relative homology groups.  
\end{lemma}
\begin{proof}
 The proof is   an  analogue of  Lemma~\ref{le-zoey}   and   \cite[p. 118]{hatcher}.  We  omit the details.   
 \end{proof}
 It follows from  the  long  exact  sequence  in  the middle row of Lemma~\ref{eq-2.55} that the rank of the  relative embedded homology  group is sub-additive:  
 \begin{corollary}\label{co-3.66}
For any hypergraph triple $(\mathcal{H}, \mathcal{A}, \mathcal{B})$ and any $n\geq 0$, we have 
\begin{eqnarray*}
{\rm rank} H_n(\mathcal{H}, \mathcal{B}) \leq {\rm rank} H_n(\mathcal{H}, \mathcal{A})  + {\rm rank} H_n(\mathcal{A}, \mathcal{B}).  
\end{eqnarray*}
\end{corollary}
\begin{proof}
The proof is an analogue of  \cite[p. 337]{ana1}.
\end{proof}

\section{Cell structures from the embedded homology}\label{s5x}

Let  $\mathcal{H}$  be a hypergraph. 

\begin{definition}
For each $n\geq 0$,  we define the {\it $n$-skeleton} $\mathcal{H}^n$  of $\mathcal{H}$ to  be the sub-hypergraph consisting of all the hyperedges of dimensions  at most $n$.  
\end{definition}

Applying   the  diagram   in   Lemma~\ref{eq-2.55}  to the   hypergraph  triple $(\mathcal{H}^n, \mathcal{H}^{n-1}, \mathcal{H}^{n-2})$  and  with  the help of  the argument in \cite[p. 139]{hatcher},  we obtain  a commutative diagram 
{\footnotesize\begin{eqnarray}\label{eq-pfkq}
\xymatrix{
\cdots \ar[r]^-{\partial_*} 
&H_{n+1}(\Delta(\mathcal{H}^{n+1}),\Delta(\mathcal{H}^n))\ar[r]^-{\partial_*} &H_n(\Delta(\mathcal{H}^n),\Delta(\mathcal{H}^{n-1}))\ar[r]^-{\partial_*} &\\
\cdots \ar[r]^-{\partial_*} 
&H_{n+1}(\mathcal{H}^{n+1},\mathcal{H}^n)\ar[r]^-{\partial_*}\ar[u] &H_n(\mathcal{H}^n,\mathcal{H}^{n-1})\ar[r]^-{\partial_*}\ar[u] &\\
\cdots \ar[r]^-{\partial_*} 
&H_{n+1}(\delta(\mathcal{H}^{n+1}),\delta(\mathcal{H}^n))\ar[r]^-{\partial_*} \ar[u]
&H_n(\delta(\mathcal{H}^n),\delta(\mathcal{H}^{n-1}))\ar[r]^-{\partial_*} 
\ar[u] &\\
\ar[r]&H_{n-1}(\Delta(\mathcal{H}^{n-1}),\Delta(\mathcal{H}^{n-2}))\ar[r]^-{\partial_*} &\cdots &\\
\ar[r]&H_{n-1}(\mathcal{H}^{n-1},\mathcal{H}^{n-2})\ar[r]^-{\partial_*}\ar[u]
 &\cdots &\\
\ar[r] &H_{n-1}(\delta(\mathcal{H}^{n-1}),\delta(\mathcal{H}^{n-2}))\ar[r]^-{\partial_*} \ar[u]
&\cdots &
}
\end{eqnarray}}where each row is a chain complex.  For simplicity,  we write the three rows on the  last  commutative diagram as  
 three chain complexes
\begin{eqnarray}\label{eq-3.88}
&\{H_*(\delta(\mathcal{H}^*),\delta(\mathcal{H}^{*-1})), \partial _*\}
\longrightarrow \{H_*(\mathcal{H}^*, \mathcal{H}^{*-1}), \partial_*\} \nonumber\\
& \longrightarrow 
\{H_*(\Delta(\mathcal{H}^*), \Delta(\mathcal{H}^{*-1})), \partial _*\}. 
\end{eqnarray}
Here  in (\ref{eq-3.88}),  the two arrows are the chain maps given  as  the vertical maps  in the  diagram  (\ref{eq-pfkq}).  
\begin{definition}
We interpret (\ref{eq-3.88}) as the {\it cell structure} of the hypergraph $\mathcal{H}$.  
\end{definition}
Note that for a general hypergraph $\mathcal{H}$,  the first chain complex in (\ref{eq-3.88}) could  be different from the chain complex  $C_*(\Delta\mathcal{H})$ of $\Delta\mathcal{H}$, and 
the last chain complex in (\ref{eq-3.88})  could  be different from the chain complex  $C_*(\delta\mathcal{H})$  of $\delta\mathcal{H}$.  
However,   when $\mathcal{H}$ is a simplicial complex, 
all the three chain complexes in (\ref{eq-3.88}) are the same, which is  the usual chain complex of the cellular homology $\mathcal{H}$.

\begin{lemma}\label{le-cell}
For any hypergraph $\mathcal{H}$ and any $i,n\geq 0$,  we have
\begin{enumerate}[(i).]
\item
  $H_i(\mathcal{H}^n,\mathcal{H}^{n-1})=0$ if  $i\neq n$;
  \item
     $H_i(\mathcal{H}^n,\mathcal{H}^{n-1})={\rm Inf}_n(\mathcal{H})$ if $i=n$.
     \end{enumerate}   
\end{lemma}
\begin{proof}
Let $\mathcal{H}$ be a hypergraph and let $i,n\geq 0$.  We observe that
\begin{eqnarray*}
{\rm Inf}_i(\mathcal{H}^n)=\left\{
\begin{array}{ll}
{\rm Inf}_i(\mathcal{H}), & {\rm ~if~} i\leq n,\\
0, &{\rm ~if~}  i\geq n+1. 
\end{array}
\right. 
\end{eqnarray*}
Hence
\begin{eqnarray}\label{eq-2.chain}
{\rm Inf}_i(\mathcal{H}^n)/{\rm Inf}_i(\mathcal{H}^{n-1})=\left\{
\begin{array}{ll}
0, & {\rm ~if~} i\leq n-1 {\rm ~or~} i\geq n+1,\\
{\rm Inf}_n(\mathcal{H}), &{\rm ~if~}  i=n. 
\end{array}
\right. 
\end{eqnarray}
Therefore,  taking the homology of the quotient infimum chain complex (\ref{eq-2.chain}),  we have
\begin{eqnarray*}
H_i(\mathcal{H}^n,\mathcal{H}^{n-1})=\left\{
\begin{array}{ll}
0, & {\rm ~if~} i\neq n,\\
{\rm Inf}_n(\mathcal{H}), & {\rm ~if~} i= n. 
\end{array}
\right. 
\end{eqnarray*}
We  obtain  (i)  and  (ii).  
\end{proof}

For a hypergraph $\mathcal{H}$,  we  use  $\mathcal{D}(\mathcal{H})$  to  denote  the middle chain complex of (\ref{eq-3.88}) i.e. the middle row  of  the diagram (\ref{eq-pfkq}).  By Lemma~\ref{le-cell},   
 the  next  proposition follows.

\begin{proposition}\label{pr-3.3.1}
Let $\mathcal{H}$ be a hypergraph.   Then there is an isomorphism
\begin{eqnarray*}
  H_*(\mathcal{D}(\mathcal{H}))\overset{\cong}{\longrightarrow} H_*(\mathcal{H})
\end{eqnarray*}
which is natural with respect to morphisms of hypergraphs. 
\end{proposition}
\begin{proof}
The isomorphism $\lambda$ as well as its functorial property follows from Lemma~\ref{le-cell}.   Alternatively,  
we  substitute  the  simplicial complex  in the proof of \cite[Theorem~39.4]{eat} by a hypergraph  and substitute the simplicial homology  in the proof of \cite[Theorem~39.4]{eat} by the embedded homology of hypergraphs.  
We verify that all the three steps in the proof of \cite[Theorem~39.4]{eat} still hold for the embedded homology of hypergraphs.  In  this way  we will  also  prove Proposition~\ref{pr-3.3.1}.  
\end{proof}

\section{The  homology of  hypergraphs and  associated  complexes}\label{s6x}

We consider  a special case of the relative embedded homology groups. 

\begin{theorem}[Main Result I]\label{pr-xqy}
Let $\mathcal{H}$  by a hypergraph.  Then we  have a commutative diagram 
 {\footnotesize \begin{eqnarray*} 
\xymatrix{
& &\cdots \ar[d] &&&\cdots\ar[ld]\\
& &H_{n+1}(\Delta\mathcal{H},\delta\mathcal{H}) \ar[d]^-{(\partial_{n+1}'')_*} &&H_{n+1}(\Delta\mathcal{H},\mathcal{H})\ar[ld]_-{(\partial'_{n+1})_*}&\\
\cdots \ar[r] &H_{n+1}(\mathcal{H},\delta\mathcal{H})\ar[r]^-{(\partial_{n+1})_*}  &H_n(\delta\mathcal{H})\ar[r]^-{i_*} \ar[d]^-{i''_*} & H_n(\mathcal{H})\ar[r]^-{j_*}\ar[ld]_-{i'_*} & H_{n}(\mathcal{H},\delta\mathcal{H})\ar[r]^-{(\partial_n)_*} &\cdots\\
&&H_n(\Delta\mathcal{H})\ar[d]^-{j''_*}\ar[ld]_-{j'_*}&&&\\
&H_{n}(\Delta\mathcal{H},\mathcal{H})\ar[ld]_-{(\partial'_n)_*}&H_n(\Delta\mathcal{H},\delta\mathcal{H})\ar[d]^-{(\partial''_n)_*}&&&\\
\cdots&&\cdots&&&
}
\end{eqnarray*}}where  
\begin{enumerate}[(i).]
\item 
all 
the vertical maps give a long exact sequence,
\item
all the horizontal maps give a long exact sequence,
\item
all the skew maps (the arrows from  right-up to  left-down) give a long exact sequence. 
\end{enumerate} 
\end{theorem}
\begin{proof}
  Given  a hypergraph $\mathcal{H}$,  we have the following commutative diagram of injective morphisms of  hypergraph pairs
{\footnotesize\begin{eqnarray*}
\xymatrix{
(\delta\mathcal{H},\emptyset)\ar[d]  & &\\
(\mathcal{H},\emptyset)\ar[d] \ar[r]  &(\mathcal{H},\delta\mathcal{H})\ar[d] &\\
(\Delta\mathcal{H},\emptyset)\ar[r]  &(\Delta\mathcal{H},\delta\mathcal{H}) \ar[r] & (\Delta\mathcal{H},\mathcal{H}). 
}
\end{eqnarray*}}By  the second row of the diagram Lemma~\ref{le-zoey}  in Lemma~\ref{le-zoey},  the  last diagram  induces a commutative diagram of (relative) embedded homology groups 
{\footnotesize \begin{eqnarray*}
\xymatrix{
H_*(\delta\mathcal{H})\ar[d]_-{i_*}\ar@/_2pc/[dd]_-{i''_*}  && &&\\
H_*(\mathcal{H})\ar[d]_-{i'_*} \ar[rr]^-{j_*}  &&H_*(\mathcal{H},\delta\mathcal{H})\ar[llu]_-{\partial_*}\ar[d] &&\\
H_*(\Delta\mathcal{H})\ar[rr]^-{j''_*}\ar@/_2pc/[rrrr]^-{j'_*}  &&H_*(\Delta\mathcal{H},\delta\mathcal{H})\ar@/^1pc/[lluu]^-{\partial''_*} \ar[rr] && H_*(\Delta\mathcal{H},\mathcal{H})\ar@/_3pc/[llllu]_-{\partial'_*}. 
}
\end{eqnarray*}}Here each of the triple  $(i_*, j_*, \partial_*)$,  $(i'_*, j'_*, \partial'_*)$,  and  $(i''_*, j''_*, \partial''_*)$  gives  a long exact  sequence of  (relative) embedded homology groups.  Specifically, these three long exact sequences can be expressed in the commutative diagram  Theorem~\ref{pr-xqy}.  We  finish the proof. 
\end{proof}

The next corollary follows from Theorem~\ref{pr-xqy}.  

\begin{corollary}\label{th1}
Let $m\geq l+1$.  Suppose $H_n(\mathcal{H})=0$  for any $l\leq n\leq m$.  Then for any $l+1\leq n\leq m$,  we have a short exact sequence
{\footnotesize\begin{eqnarray*}
\xymatrix{
0\ar[r] & H_n(\Delta\mathcal{H})\ar[d]^{\cong} \ar[r] &\tilde H_n(\Delta\mathcal{H}/\delta\mathcal{H})\ar[r] &H_{n-1} (\delta\mathcal{H})\ar[r] \ar[d] ^{\cong}  &0  \\
&H_n(\Delta\mathcal{H},\mathcal{H})&&H_n(\mathcal{H},\delta\mathcal{H}). &
}
\end{eqnarray*} }
\end{corollary}
\begin{remark}
Here we abuse the notation $\Delta\mathcal{H}/\delta\mathcal{H}$  for the quotient topological space $|\Delta\mathcal{H}|/|\delta\mathcal{H}|$ where the ambient space $|\Delta\mathcal{H}|$  is the geometric realization of $\Delta\mathcal{H}$  and the  sub-space $|\delta\mathcal{H}|$   is  the geometric realization of $\delta\mathcal{H}$.  
\end{remark}
\begin{proof}[Proof of Corollary~\ref{th1}]
By the  commutative diagram Theorem~\ref{pr-xqy} of three long exact sequences of (relative) embedded  homology groups, we have
\begin{enumerate}[(i).]
\item
$H_n(\Delta\mathcal{H},\mathcal{H})\cong H_n(\Delta\mathcal{H})$ for all $l+1\leq n\leq m$;
\item
$H_n(\mathcal{H},\delta\mathcal{H})\cong H_{n-1}(\delta\mathcal{H})$ for all $l+1\leq n\leq m$;
\item
$i''_*: H_n(\delta\mathcal{H})\longrightarrow H_n(\Delta\mathcal{H})$ is a zero-map for all $l\leq n\leq m$, which implies a short exact sequence
\begin{eqnarray*}
\xymatrix{
0\ar[r] & H_n(\Delta\mathcal{H})\ar[r]^{j''_*} & H_n(\Delta\mathcal{H},\delta\mathcal{H})\ar[r]^{\partial''_*} & H_{n-1}(\delta\mathcal{H}) \ar[r] &0
}
\end{eqnarray*}
for all $l+1\leq m\leq n$. 
\end{enumerate}
Since $H_n(\Delta\mathcal{H},\delta\mathcal{H})$  is the usual relative homology  group  of simplicial complex pairs,  it is isomorphic to the reduced homology $\tilde H_n(\Delta\mathcal{H}/\delta\mathcal{H})$ of the quotient space $\Delta\mathcal{H}/\delta\mathcal{H}$.     Thus summarizing (i), (ii)  and (iii),  we obtain the corollary. 
\end{proof}

Supplementary to Theorem~\ref{pr-xqy}  and Corollary~\ref{th1}, we apply the middle row in diagram Lemma~\ref{eq-2.55} to the triple $(\Delta\mathcal{H},\mathcal{H},\delta\mathcal{H})$.  We obtain a long exact sequence  of relative embedded homology groups 
\begin{eqnarray*} 
 \cdots\longrightarrow    H_n(\mathcal{H},\delta\mathcal{H}) \overset{{(i_*)_n}}{\longrightarrow}   H_n( \Delta\mathcal{H},\delta\mathcal{H})  \overset{(j_*)_n}{\longrightarrow }   
 H_n(\Delta\mathcal{H},\mathcal{H}) \overset{(\partial_*)_n}{\longrightarrow }   H_{n-1}(\mathcal{H},\delta\mathcal{H}) \longrightarrow \cdots
\end{eqnarray*}
It follows with the help of  Corollary~\ref{co-3.66} that
\begin{eqnarray*}
{\rm rank} H_n(\Delta\mathcal{H}, \delta\mathcal{H}) \leq {\rm rank} H_n(\Delta\mathcal{H}, \mathcal{H})  + {\rm rank} H_n(\mathcal{H}, \delta\mathcal{H}).  
\end{eqnarray*}

\section{The  Mayer-Vietoris sequence for hypergraph pairs}\label{s7x}

The relative  version  of  the Mayer-Vietoris sequence for  the  homology  of   simplicial complex pairs can be found in \cite[Section 4.6]{spa}.   On  the  other hand,  the Mayer-Vietoris sequence  for  the embedded homology of hypergraphs  is given in \cite[Section 3.3]{hg1}.   In  this  section,  we will prove a relative version of  the Mayer-Vietoris sequence  for  the embedded homology   of hypergraph pairs.

Let $(\mathcal{H},\mathcal{A})$  and  $(\mathcal{H}',\mathcal{A}')$  be two hypergraph pairs.  We  consider the following three  conditions:

\begin{enumerate}[(I).]
\item
for any $\sigma\in\mathcal{H}$ and any $\sigma'\in \mathcal{H}'$, either $\sigma\cap\sigma'$  is  the  empty-set  or  $\sigma\cap\sigma'\in \mathcal{H}\cap\mathcal{H}'$;  
\item
 for any $\tau\in\mathcal{A}$ and any $\tau'\in \mathcal{A}'$, either $\tau\cap\tau'$  is the empty-set  or  $\tau\cap\tau'\in \mathcal{A}\cap\mathcal{A}'$;
 \item
  both  $\mathcal{H}\cap\mathcal{H}'$ and $\mathcal{A}\cap\mathcal{A}'$ are  disjoint unions of standard simplicial complexes
\begin{eqnarray*}
\mathcal{H}\cap\mathcal{H}'=\sqcup_{i=1}^k \Delta[n_i],~~~~~~
\mathcal{A}\cap\mathcal{A}'=\sqcup_{i=1}^k \Delta[m_i]
\end{eqnarray*}
where for each $1\leq i\leq k$,  it holds that $m_i\leq n_i$ and $\Delta[m_i]$  is a simplicial sub-complex  of $\Delta[n_i]$.  
 \end{enumerate}

\begin{theorem}[Main Result II]\label{th-3.888}
Let $(\mathcal{H},\mathcal{A})$  and  $(\mathcal{H}',\mathcal{A}')$  be two hypergraph pairs such that    both (I)  and  (II)  are satisfied.  
 Then we have a long exact sequence of the  relative embedded homology groups 
\begin{eqnarray}\label{eq-zxva}
&\cdots\longrightarrow H_n(\mathcal{H}\cap\mathcal{H}', \mathcal{A}\cap\mathcal{A}')\longrightarrow H_n(\mathcal{H},\mathcal{A})\oplus H_n(\mathcal{H}',\mathcal{A}')\longrightarrow \nonumber \\
&H_n(\mathcal{H}\cup\mathcal{H}',\mathcal{A}\cup\mathcal{A}')\longrightarrow H_{n-1}(\mathcal{H}\cap\mathcal{H}', \mathcal{A}\cap\mathcal{A}')\longrightarrow \cdots
\end{eqnarray} 
\end{theorem}

\begin{proof}
Let  $(\mathcal{H},\mathcal{A})$  and  $(\mathcal{H}',\mathcal{A}')$  be two  hypergraph pairs  such  that  both (I)  and  (II)  are  satisfied.    Then with the help of  \cite[Proposition~3.9]{hg1}, we have a commutative diagram
{\footnotesize\begin{eqnarray*}
\xymatrix{
0\ar[r] &{\rm Inf}_*(\mathcal{A}\cap\mathcal{A}')\ar[r]\ar[d] &{\rm Inf}_*(\mathcal{A})\oplus {\rm Inf}_*(\mathcal{A}')\ar[r]\ar[d]  & {\rm Inf}_*(\mathcal{A}\cup \mathcal{A}')\ar[r]\ar[d] &0\\
0\ar[r] &{\rm Inf}_*(\mathcal{H}\cap\mathcal{H}')\ar[r] &{\rm Inf}_*(\mathcal{H})\oplus {\rm Inf}_*(\mathcal{H}')\ar[r]  & {\rm Inf}_*(\mathcal{H}\cup \mathcal{H}')\ar[r] &0
}
\end{eqnarray*}}of chain complexes  
where both rows are short exact sequences and all the vertical maps are canonical inclusions.    
By taking the quotient chain complex  with respect to each vertical map in the last diagram,  we  have a short exact sequence 
\begin{eqnarray}\label{eq-xyza}
&0\longrightarrow {\rm Inf}_*(\mathcal{H}\cap\mathcal{H}')/{\rm Inf}_*(\mathcal{A}\cap\mathcal{A}')\longrightarrow {\rm Inf}_*(\mathcal{H})/{\rm Inf}_*(\mathcal{A})\oplus {\rm Inf}_*(\mathcal{H}')/{\rm Inf}_*(\mathcal{A}')\nonumber\\
&\longrightarrow {\rm Inf}_*(\mathcal{H}\cup \mathcal{H}')/{\rm Inf}_*(\mathcal{A}\cup \mathcal{A}')\longrightarrow 0 
\end{eqnarray}
of  chain  complexes.  
Applying the homology functor to the short exact sequence (\ref{eq-xyza})   of  chain  complexes,  we  obtain  the long exact sequence  (\ref{eq-zxva}) of  the relative embedded homology groups.  
\end{proof}

Theorem~\ref{th-3.888} yields  the next corollary.   

\begin{corollary}\label{co-qmboz}
Let $\mathcal{H}$ and $\mathcal{H}'$ be two hypergraphs such that  (I)  is satisfied.  Then we have two long exact sequences of the  relative homology groups 
\begin{eqnarray*}
&\cdots\longrightarrow H_n(\mathcal{H}\cap\mathcal{H}', \delta\mathcal{H}\cap\delta\mathcal{H}')\longrightarrow H_n(\mathcal{H},\delta\mathcal{H})\oplus H_n(\mathcal{H}',\delta\mathcal{H}')\\
&\longrightarrow H_n(\mathcal{H}\cup\mathcal{H}',\delta\mathcal{H}\cup\delta\mathcal{H}')\longrightarrow H_{n-1}(\mathcal{H}\cap\mathcal{H}', \delta\mathcal{H}\cap\delta\mathcal{H}')\longrightarrow \cdots
\end{eqnarray*}
and
\begin{eqnarray*}
&\cdots\longrightarrow H_n(\Delta\mathcal{H}\cap\Delta\mathcal{H}', \mathcal{H}\cap\mathcal{H}')\longrightarrow H_n(\Delta\mathcal{H},\mathcal{H})\oplus H_n(\Delta\mathcal{H}',\mathcal{H}')\\
&\longrightarrow H_n(\Delta\mathcal{H}\cup\Delta\mathcal{H}',\mathcal{H}\cup\mathcal{H}')\longrightarrow H_{n-1}(\Delta\mathcal{H}\cap\Delta\mathcal{H}', \mathcal{H}\cap\mathcal{H}')\longrightarrow \cdots
\end{eqnarray*}
\end{corollary}
\begin{proof}
Let $\eta\in \delta\mathcal{H}$  and $\eta'\in\delta\mathcal{H}'$.  Then for any non-empty subset   $\sigma\subseteq\eta$  and   any  non-empty subset  $\sigma'\subseteq  \eta'$,  
we  have  $\sigma\in\mathcal{H}$  and  $\sigma'\in\mathcal{H}'$.   For any non-empty subset $\kappa\subseteq\eta\cap\eta'$,  we  have that $\kappa\subseteq\eta$  and $\kappa\subseteq\eta'$.  Thus $\kappa\in\mathcal{H}$  and  $\kappa\in \mathcal{H}'$.  That is,  $\kappa\in\mathcal{H}\cap\mathcal{H}'$.  This implies 
\begin{eqnarray}\label{eq-oin}
\eta\cap\eta'\in \delta(\mathcal{H}\cap\mathcal{H}').  
\end{eqnarray}
Thus  by Lemma~\ref{le-int-union}~(i)  and   (\ref{eq-oin})    we  have $\tau\cap\tau'\in  \delta\mathcal{H}\cap\delta\mathcal{H}'$.   Consequently,  the  hypergraph pairs  $(\mathcal{H}, \delta\mathcal{H})$  and   $(\mathcal{H}', \delta\mathcal{H}')$  satisfy (I) and (II).  
 Therefore, by 
substituting $(\mathcal{H},\mathcal{A})$ with $(\mathcal{H},\delta\mathcal{H})$ and substituting $(\mathcal{H}',\mathcal{A}')$ with $(\mathcal{H}',\delta\mathcal{H}')$ in  Theorem~\ref{th-3.888},   we  obtain 
the first long exact sequence.

Let $\tau\in \Delta\mathcal{H}$  and $\tau'\in\Delta\mathcal{H}'$.  Then there exist $\sigma\in\mathcal{H}$   and $\sigma'\in\mathcal{H}'$  such  that $\tau\subseteq\sigma$  and  $\tau'\in\sigma'$.   Suppose that $\tau\cap\tau'\neq\emptyset$.  Then  since $\tau\cap\tau'$  is a subset of $\sigma\cap\sigma'$,  it follows that $\sigma\cap\sigma'\neq\emptyset$.  With the help of (I),  we  have that $\sigma\cap\sigma'\in\mathcal{H}\cap\mathcal{H}'$.  Consequently,  
\begin{eqnarray}\label{eq-za}
\tau\cap\tau'\in  \Delta(\mathcal{H}\cap\mathcal{H}').   
\end{eqnarray} 
Thus by Lemma~\ref{le-int-union}~(iii)  and  (\ref{eq-za})   we  have  $\tau\cap\tau'\in  \Delta\mathcal{H}\cap\Delta\mathcal{H}'$.  Consequently,  the  hypergraph pairs  $(\Delta\mathcal{H},\mathcal{H})$  and   $(\Delta\mathcal{H}',\mathcal{H}')$  satisfy (I) and (II).  
Therefore,  by  substituting $(\mathcal{H},\mathcal{A})$ with  $(\Delta\mathcal{H},\mathcal{H})$  and substituting $(\mathcal{H}',\mathcal{A}')$ with  $(\Delta\mathcal{H}',\mathcal{H}')$  in  Theorem~\ref{th-3.888},   we obtain the  second long exact sequence.  
\end{proof}

We have the following example. 

\begin{example}
(i).  Let  $(\mathcal{H},\mathcal{A})$  and $(\mathcal{H}',\mathcal{A}')$  be two hypergraph pairs satisfying (I), (II) and  (III).     
Then the quotient space $\Delta[n_i]/\Delta[m_i]=|\Delta[n_i]|/|\Delta[m_i]|$  is contractible for each $1\leq i\leq k$,  which implies 
\begin{eqnarray*}
H_n(\mathcal{H}\cap\mathcal{H}',\mathcal{A}\cap\mathcal{A}')=0, ~~~ n\geq 1. 
\end{eqnarray*}
  Therefore,  by Theorem~\ref{th-3.888}  we have
\begin{eqnarray*}
H_n(\mathcal{H}\cup\mathcal{H}',\mathcal{A}\cup\mathcal{A}')\cong H_n(\mathcal{H},\mathcal{A})\oplus H_n(\mathcal{H}',\mathcal{A}'),  ~~~n\geq 2. 
\end{eqnarray*}

\noindent  (ii).   Let $(\mathcal{H}(j),\mathcal{A}(j))$,  $1\leq j\leq m$,  be a sequence   of hypergraph pairs such that for any $1\leq j_1<j_2\leq m$,  the pair  $(\mathcal{H}(j_1),\mathcal{A}(j_1))$ and $(\mathcal{H}(j_2),\mathcal{A}(j_2))$ satisfy  (I), (II) and (III).  By an induction on $m$,   we have 
\begin{eqnarray*}
H_n(\cup_{j=1}^m \mathcal{H}(j), \cup_{j=1}^m \mathcal{A}(j))= \oplus_{j=1}^m H_n(\mathcal{H}(j),\mathcal{A}(j)), ~~~ n\geq 2. 
\end{eqnarray*}

\noindent  (iii).   Let $V=\{v_0,v_1,v_2,v_3, w_0,w_1,w_2,w_3\}$  be a set of  eight  vertices.    Consider  the closed  tetrahedron
\begin{eqnarray*}
\Delta[3]&=&\{ \{v_0,v_1,v_2,v_3\},\{v_0,v_1,v_2\},\{v_0,v_1,v_3\},\{v_0,v_2,v_3\},\{v_1,v_2,v_3\}, \\
&&\{v_0,v_1\},\{v_0,v_2\},\{v_0,v_3\}, \{v_1,v_2\},\{v_1,v_3\},\{v_2,v_3\},\{v_0\},\{v_1\},\{v_2\},\{v_3\}\}
\end{eqnarray*}
 where   $v_0,v_1,v_2,v_3 $ are the distinct  four vertices of $\Delta[3]$.  
    Consider the eight hypergraphs
  \begin{eqnarray*}
  &\mathcal{H}(0) =  \{\{w_0,v_1,v_2,v_3\}, \{w_0,v_1,v_2\}, \{w_0,v_1,v_3\},\{w_0,v_2,v_3\}\} \cup \Delta[3],  \\
  & \mathcal{H}(1) = \{\{v_0,w_1,v_2,v_3\}, \{v_0,w_1,v_2\}, \{v_0,w_1,v_3\},\{w_1,v_2,v_3\}\}  \cup \Delta[3], \\
    & \mathcal{H}(2) = \{\{v_0,v_1,w_2,v_3\}, \{v_0,v_1,w_2\}, \{v_0,w_2,v_3\},\{v_1,w_2,v_3\}\}  \cup \Delta[3], \\
    & \mathcal{H}(3) = \{\{v_0,v_1,v_2,v_3\},  \{v_0,v_1,w_3\},\{v_0,v_2,w_3\},\{v_1,v_2,w_3\}\}  \cup \Delta[3], \\
     &\mathcal{A}(0) =  \{\{w_0,v_1,v_2\}, \{w_0,v_1,v_3\},\{w_0,v_2,v_3\}\} \cup \Delta[3],  \\
  & \mathcal{A}(1) = \{\{v_0,w_1,v_2\}, \{v_0,w_1,v_3\},\{w_1,v_2,v_3\}\}  \cup \Delta[3], \\
    & \mathcal{A}(2) = \{\{v_0,v_1,w_2\}, \{v_0,w_2,v_3\},\{v_1,w_2,v_3\}\}  \cup \Delta[3], \\
    & \mathcal{A}(3) = \{\{v_0,v_1,w_3\},\{v_0,v_2,w_3\},\{v_1,v_2,w_3\}\}  \cup \Delta[3].
      \end{eqnarray*}
For any $0\leq i<j\leq 3$,  it can be verified that $(\mathcal{H}(i), \mathcal{A}(i))$  and  $(\mathcal{H}(j), \mathcal{A}(j))$  satisfy (a), (b), and (c) in (i).  
  Consequently, for $n\geq 2$,   it follows from (ii)  that  
\begin{eqnarray}\label{eq-ex.mvsq}
H_n(\cup_{i=0}^3 \mathcal{H}(i), \cup_{i=0}^3\mathcal{A}(i))=\oplus_{i=0}^3 H_n(\mathcal{H}(i),\mathcal{A}(i))
= H_n(\mathcal{H}(0),\mathcal{A}(0))^{\oplus 4}. 
\end{eqnarray} 
On the other hand,  
\begin{eqnarray*}
&{\rm Inf}_3(\mathcal{H}(0))={\rm Sup}_3(\mathcal{H}(0))=\mathbb{Z}(\{v_0,v_1,v_2,v_3\},\{w_0,v_1,v_2,v_3\}),\\
 &{\rm Inf}_3(\mathcal{A}(0))={\rm Sup}_3(\mathcal{A}(0))=0,\\
 &{\rm Inf}_2(\mathcal{H}(0))={\rm Sup}_2(\mathcal{H}(0))={\rm Inf}_2(\mathcal{A}(0))={\rm Sup}_2(\mathcal{A}(0))\\
 &=\mathbb{Z}(\{v_0,v_1,v_2\},\{v_0,v_1,v_3\}, 
  \{v_0,v_2,v_3\},\{v_1,v_2,v_3\}, \\
  &\{v_0,v_1,v_3\}-\{w_0,v_2,v_3\}+ \{w_0,v_1,v_3\}-\{w_0,v_1,v_2\} ),\\
 &{\rm Inf}_1(\mathcal{H}(0))={\rm Inf}_1(\mathcal{A}(0))\\
 &=\mathbb{Z}(\{v_0,v_1\},\{v_0,v_2\},\{v_0,v_3\}, \{v_1,v_2\},\{v_1,v_3\},\{v_2,v_3\}),\\
 &{\rm  Sup}_1(\mathcal{H}(0))={\rm Sup}_1(\mathcal{A}(0))\\
 &=\mathbb{Z}(\{v_0,v_1\},\{v_0,v_2\},\{v_0,v_3\}, \{v_1,v_2\},\{v_1,v_3\},\{v_2,v_3\},\\
 &\{v_1,v_2\}-\{w_0,v_2\}+\{w_0,v_1\},\{v_1,v_3\}-\{w_0,v_3\}+\{w_0,v_1\},\\
 &\{v_2,v_3\}-\{w_0,v_3\}+\{w_0,v_2\}),\\
 &{\rm Inf}_0(\mathcal{H}(0))={\rm Sup}_0(\mathcal{H}(0))={\rm Inf}_0(\mathcal{A}(0))={\rm Sup}_0(\mathcal{A}(0))\\
 &=\mathbb{Z}(\{v_0\},\{v_1\},\{v_2\},\{v_3\})   
\end{eqnarray*}
which implies
\begin{eqnarray*}
&{\rm Inf}_3(\mathcal{H}(0))/{\rm Inf}_3(\mathcal{A}(0))={\rm Sup}_3(\mathcal{H}(0))/{\rm sup}_3(\mathcal{A}(0))\\
&=\mathbb{Z}(\{v_0,v_1,v_2,v_3\},\{w_0,v_1,v_2,v_3\}), \\
& {\rm Inf}_n(\mathcal{H}(0))/{\rm Inf}_n(\mathcal{A}(0))={\rm Sup}_n(\mathcal{H}(0))/{\rm  Sup}_n(\mathcal{A}(0))=0 ~{\rm ~for~}n\neq  3. 
\end{eqnarray*}
Hence
\begin{eqnarray*}
H_n(\mathcal{H}(0),\mathcal{A}(0))=\left\{
\begin{array}{ll}
0, & n=0,1,2,\\
\mathbb{Z}\oplus\mathbb{Z}, &n=3. 
\end{array}
\right. 
\end{eqnarray*}
Therefore,   if  we let 
$\mathcal{H}=\cup_{i=0}^3 \mathcal{H}(i)$  and   
$\mathcal{A}=\cup_{i=0}^3\mathcal{A}(i)$,  
then with the help of (\ref{eq-ex.mvsq}) we have
\begin{eqnarray*}
H_n(\mathcal{H},\mathcal{A})=\left\{
\begin{array}{ll}
0, & n=2,\\
\mathbb{Z}^{\oplus 8}, &n=3. 
\end{array}
\right. 
\end{eqnarray*}
\end{example}

\section{Two-dimensional  persistence  of  the relative  embedded  homology of    hypergraph  pairs}\label{s8x}

Let $f: \mathcal{H}\longrightarrow \mathbb{R}$ be a real valued function on a hypergraph $\mathcal{H}$.  Then $f$ assigns  a real number $f(\sigma)$ to each hyperedge $\sigma$.  For any $t\in \mathbb{R}$,   the {\it level hypergraph} is 
$\mathcal{H}(t)=\{\sigma\in\mathcal{H}\mid f(\sigma)\leq t\}$.   For any real numbers $a\leq b$,  we have a hypergraph pair $(\mathcal{H}(b), \mathcal{H}(a))$.   The two inclusions 
\begin{eqnarray*}
(\delta(\mathcal{H}(b)), \delta(\mathcal{H}(a)))\longrightarrow (\mathcal{H}(b), \mathcal{H}(a))\longrightarrow (\Delta(\mathcal{H}(b)), \Delta(\mathcal{H}(a)))
\end{eqnarray*}
of hypergraph pairs (or simplicial complex pairs)  induce two homomorphisms
\begin{eqnarray*}
H_*(\delta(\mathcal{H}(b)), \delta(\mathcal{H}(a)))\longrightarrow H_*(\mathcal{H}(b), \mathcal{H}(a))\longrightarrow H_*(\Delta(\mathcal{H}(b)), \Delta(\mathcal{H}(a)))
\end{eqnarray*}
of the relative  (embedded)  homology groups.   For any two points  $(x,y), (x',y')\in \mathbb{R}^2$,  we write $(x,y)\leq (x',y')$ if and only if  $x\leq y$ and $x'\leq y'$.  
Let $a\leq b$ and $a'\leq b'$ with $(a,b)\leq (a',b')$.  We have a commutative diagram 
{\footnotesize\begin{eqnarray*}
\xymatrix{
(\delta(\mathcal{H}(b)), \delta(\mathcal{H}(a)))\ar[r]\ar[d] &(\mathcal{H}(b), \mathcal{H}(a))\ar[r] \ar[d]&(\Delta(\mathcal{H}(b)), \Delta(\mathcal{H}(a)))\ar[d]\\
(\delta(\mathcal{H}(b')), \delta(\mathcal{H}(a')))\ar[r] &(\mathcal{H}(b'), \mathcal{H}(a'))\ar[r] &(\Delta(\mathcal{H}(b')), \Delta(\mathcal{H}(a')))
}
\end{eqnarray*}}of hypergraph pairs where each arrow is an injection.  This induces a commutative diagram of relative (embedded)
  homology groups 
 {\footnotesize \begin{eqnarray}\label{eq-monzmfe}
\xymatrix{
H_*(\delta(\mathcal{H}(b)), \delta(\mathcal{H}(a)))\ar[r]\ar[d] &H_*(\mathcal{H}(b), \mathcal{H}(a))\ar[r] \ar[d]&H_*(\Delta(\mathcal{H}(b)), \Delta(\mathcal{H}(a)))\ar[d]\\
H_*(\delta(\mathcal{H}(b')), \delta(\mathcal{H}(a')))\ar[r] &H_*(\mathcal{H}(b'), \mathcal{H}(a'))\ar[r] &H_*(\Delta(\mathcal{H}(b')), \Delta(\mathcal{H}(a'))). 
}
\end{eqnarray}}
 
By the definition of multi-dimensional persistence  homology (cf. \cite[Definition~10]{m3},  \cite{m1},  and \cite[Subsection~2.1]{m2}),  Corollary~\ref{co-3.66},  and   \cite[p. 337]{ana1},   the next proposition  follows.  
 \begin{proposition}\label{pr-8.1}
Let $\mathcal{H}$ be a hypergraph and  $f: \mathcal{H}\longrightarrow \mathbb{R}$ be a real valued function on $\mathcal{H}$.  Then we have  a sequence of two-dimensional persistence  modules
\begin{eqnarray}\label{eq-zomdqnmt}
& \{H_*(\delta(\mathcal{H}(b)), \delta(\mathcal{H}(a)))\}_{a\leq b}
 \longrightarrow \{H_*(\mathcal{H}(b), \mathcal{H}(a))\}_{ a\leq b} \nonumber\\   
 &\longrightarrow \{H_*(\Delta(\mathcal{H}(b)), \Delta(\mathcal{H}(a)))\}_{ a\leq b}
\end{eqnarray}
where each arrow is a persistence homomorphism between persistent modules.  
Moreover,  for any $a\leq b\leq c$ and any $n\geq 0$  we have
\begin{eqnarray*}
{\rm rank} H_n(\mathcal{H}(c), \mathcal{H}(a))&\leq& {\rm rank} H_n(\mathcal{H}(c),  \mathcal{H}(b)) \\
&&+ {\rm rank} H_n(\mathcal{H}(b),  \mathcal{H}(a)),\\
{\rm rank} H_n(\delta(\mathcal{H}(c)), \delta(\mathcal{H}(a)))&\leq& {\rm rank} H_n(\delta(\mathcal{H}(c)),  \delta(\mathcal{H}(b)))\\
&& + {\rm rank} H_n(\delta(\mathcal{H}(b)),  \delta(\mathcal{H}(a))),\\
{\rm rank} H_n(\Delta(\mathcal{H}(c)), \Delta(\mathcal{H}(a)))&\leq& {\rm rank} H_n(\Delta(\mathcal{H}(c)),  \Delta(\mathcal{H}(b))) \\
&&+ {\rm rank} H_n(\Delta(\mathcal{H}(b)),  \Delta(\mathcal{H}(a))).  
\end{eqnarray*}
\end{proposition}
\begin{proof}
The persistence  modules and  persistence homomorphisms in (\ref{eq-zomdqnmt})  follow from  the diagram (\ref{eq-monzmfe})  and the definition of persistence modules and  persistence homomorphisms  (cf.   \cite[Definition~10]{m3},  \cite{m1},  and \cite[Subsection~2.1]{m2}).      The  three inequalities  follow from  Corollary~\ref{co-3.66}  and   \cite[p. 337]{ana1}.   
\end{proof}

We prospect that the multi-dimensional persistence  homology  theory  may be  applied to the  relative  (embedded) homology  groups of  hypergraph pairs  to give  a potential tool in  the  data analytics of the  hypergraph-type complex networks.

\section*{Acknowledgement} {The  authors  would like to express  their  deep
gratitude to the referee for the careful reading of the manuscript.}

 Shiquan Ren  
 
 Address:  School of Mathematics and Statistics, Henan University,  Kaifeng,  475004,  China.
 
  e-mail:  renshiquan@henu.edu.cn

\medskip

 Jie  Wu 
 Address: Yanqi Lake Beijing Institute of Mathematical Sciences and Applications, Beijing, 101408, China.
 
  e-mail: wujie@bimsa.cn

\medskip

 Mengmeng  Zhang 

Address: School of Mathematical Sciences, Hebei Normal University, Shijiazhuang, 050024, P. R. China;
Yanqi Lake Beijing Institute of Mathematical Sciences and Applications, Beijing, 101408, P. R. China.

e-mail:  mm\_zhang00@163.com

\end{document}